\documentclass[11pt,a4paper]{amsart}
\usepackage[utf8]{inputenc}
\usepackage{amsmath,amsthm}
\usepackage{epsf}
\usepackage{graphicx}
\usepackage{txfonts}
\usepackage{verbatim}
\usepackage{microtype}

\usepackage[usenames,dvipsnames]{xcolor}
\usepackage{tikz}
\usetikzlibrary{matrix,arrows}

\usepackage{mathtools}

\usepackage[T1]{fontenc}   %% for LaTeX Font Warning

\newcommand{\Fb}{\mathbb{F}}
\newcommand{\Bbf}{\mathbf{B}}
\newcommand{\Ccal}{\mathcal{C}}
\newcommand{\Ecal}{\mathcal{E}}

\newcommand{\Pcal}{\mathcal{P}}
\newcommand{\Rcal}{\mathcal{R}}
\newcommand{\Scal}{\mathcal{S}}
\newcommand{\Xcal}{\mathcal{X}}
\newcommand{\Ycal}{\mathcal{Y}}
\newcommand{\Zcal}{\mathcal{Z}}

\newcommand{\Rb}{\mathbb{R}}
\DeclareMathOperator*{\bigtimes}{\textnormal{\Large $\times$}} % Cartesian Product
\DeclareMathOperator*{\supp}{supp}  % sign operator
\theoremstyle{plain}% default
\newtheorem{lemma}{Lemma}
\newtheorem{thm}[lemma]{Theorem}
\newtheorem{prop}[lemma]{Proposition}
\newtheorem{cor}[lemma]{Corollary}

\theoremstyle{definition}
\newtheorem{defi}[lemma]{Definition}

\newtheorem{ex}[lemma]{Example}

\theoremstyle{remark}
\newtheorem{rem}[lemma]{Remark}

\newcommand{\CI}[2]{\left.#1 \Perp #2 \CIhelper}
\newcommand{\CIhelper}[1][\relax]{\def\argone{#1}\def\rrrrrrrelax{\relax}
  \ifx\argone\rrrrrrrelax\right.\else\,\middle|\,#1\right.{}\fi}

\newcommand{\ol}{\overline}

\newcommand{\pin}{p_\text{\textup{in}}}     %% the input distribution
\newcommand{\Xvin}{X_\text{\textup{in}}}    %% the input (as a random variable)
\newcommand{\Xin}{\Xcal_\text{\textup{in}}} %% the input state space
\newcommand{\Xtot}{\Xcal}                   %% the total state space

\newcommand{\fourcubevert}{
    \path (0,0) coordinate (X0);
    \path (2,0) coordinate (X1);
    \path (0,2) coordinate (X2);
    \path (2,2) coordinate (X3);
    \path (0.4,0.3) coordinate (X4);
    \path (2.4,0.3) coordinate (X5);
    \path (0.4,2.3) coordinate (X6);
    \path (2.4,2.3) coordinate (X7);

    \path (0.6,0.5) coordinate (X8);
    \path (1.4,0.5) coordinate (X9);
    \path (0.6,1.3) coordinate (X10);
    \path (1.4,1.3) coordinate (X11);
    \path (1.0,0.8) coordinate (X12);
    \path (1.8,0.8) coordinate (X13);
    \path (1.0,1.6) coordinate (X14);
    \path (1.8,1.6) coordinate (X15);}

\title{Robustness, Canalyzing functions and Systems design} % Ideals

\author{Johannes Rauh}
\author{Nihat Ay}

\email{\{rauh, nay\}@mis.mpg.de}

\address[J.~Rauh\and N.~Ay]{Max Planck Institute for Mathematics in the Sciences,
Inselstrasse 22, 
D-04103 Leipzig,Germany}
\address[N.~Ay]{Santa Fe Institute,
1399 Hyde Park Road,
Santa Fe, New Mexico 87501, USA}

\date{\today}

\keywords{robustness, conditional independence, Markov kernels}  % , primary decomposition, binomial ideals

\begin{document}

\begin{abstract}
  We study a notion of robustness of a Markov kernel that describes a system of several % $n$
  input random variables and one output random variable.  Robustness requires that the behaviour of the system does not
  change if one or several of the input variables are knocked out.  If the system is required to be robust against too
  many knockouts, then the output variable cannot distinguish reliably between input states and must be independent of
  the input.  We study how many input states the output variable can distinguish as a function of the required level of
  robustness.
  
  Gibbs potentials allow a mechanistic description of the behaviour of the system after knockouts.  Robustness imposes
  structural constraints on these potentials.  We show that interaction families of Gibbs potentials allow to describe
  robust systems.

  Given a distribution of the input random variables and the Markov kernel describing the system, we obtain a joint
  probability distribution.  Robustness implies a number of conditional independence statements for this joint
  distribution.  The set of all probability distributions corresponding to robust systems can be decomposed into a
  finite union of components, and we find parametrizations of the components.
  The decomposition corresponds to a primary decomposition of the conditional independence ideal and can be derived from
  more general results about generalized binomial edge ideals.
\end{abstract}

\maketitle

\section{Introduction}
\label{sec:introduction}

% In this article we study a notion of robustness of Markov kernels and probability distributions.
Consider a stochastic system of $n$ input nodes and one output node:

%\begin{center}
  
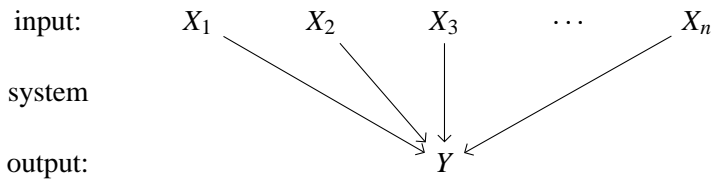
\begin{tikzpicture}
    % \foreach \i in {1,2,3} { \path (\i-3,0) coordinate (X\i); \node at (X\i) { $X_{\i}$ }; }
    % \draw (X1) -- (X2);
    \matrix (m) [matrix of math nodes, text height=1.5ex, text depth=0.25ex, column sep=2.5em, row sep=1em]
    { \text{input:} & |(X1)| X_1 & |(X2)| X_2 & |(X3)| X_3 & \cdots & |(X4)| X_n \\
      \text{system}  & & &   & &  \\
      \text{output:} & & & |(Y)| Y & &  \\
    };
    \foreach \i in {1,2,3,4} { \path[-angle 90] (X\i) edge (Y); }
  \end{tikzpicture}
%\end{center}
\\
As shown in~\cite{AyKrakauer2007:Geometric_Robustness_Theory}, there are two ingredients to robustness:
\begin{enumerate}
\item If one or several of the input nodes are removed, the system behaviour should not change too much (``small
  exclusion dependence'').
\item A causal contribution of the input nodes on the output nodes.
\end{enumerate}
The second point is strictly necessary: If the behaviour of the output does not depend on the inputs at all, then it is
usually not affected by a knockout of a subset of the inputs, but this exclusion independence is trivial.

In this paper we do not use the information theoretic measures proposed
in~\cite{AyKrakauer2007:Geometric_Robustness_Theory}.  Instead, we start with a simple model of exclusion independence:
% Nevertheless, in this paper we first focus on the first point:
We study systems in which the behaviour of the output node does not change when one or more of the input nodes are
knocked out.  We formalize our robustness requirements in terms of a \emph{robustness specification}~$\Rcal$, which
consists of pairs $(R,x_{R})$, where $R$ is a subset of the inputs and $x_{R}$ is a joint state of the inputs
in~$R$.  Let $\Scal$ be a set of possible states of the input nodes.  The system is $\Rcal$-robust in~$\Scal$, if the
behaviour of the system does not change if the inputs not in $R$ are knocked out, provided that the inputs in $R$
are in the state~$x_{R}$ and the current state of all inputs belongs to~$\Scal$.

If the robustness specification $\Rcal$ is too large, or if the set $\Scal$ is too large, then in any $\Rcal$-robust
system the output does not depend on the input at all.  In general, the behaviour of the system is restricted by
robustness requirements.  Therefore, to study the causal contribution of the input nodes on the output nodes, we
investigate how varied the behaviour of a system can be, given both $\Rcal$ and~$\Scal$.  More precisely, robustness
specifications imply that the system cannot distinguish all input states, and we may ask how many states the system can
discern.  This question is related to the topic of error detecting codes, see Remark~\ref{rem:coding}.

This paper is organized as follows: Section~\ref{sec:robust-canal} contains our basic setting and definitions.  We find
several equivalent formulations of our notion of robustness.  Moreover, we study the question how many states an
$\Rcal$-robust system can distinguish.  Section~\ref{sec:canalyzing-functions} shows that our definitions generalize the
notions of canalyzing~\cite{Kauffman93:Origins_of_Order} and nested canalyzing
functions~\cite{JarrahRaposaLaubenbacher07:Nested_Canalyzing_and_other_Functions}, which have been studied before in the
context of robustness.  Section~\ref{sec:robustn-mark-kern} proposes to model the different behaviours of a system under
various knockouts using a family of Gibbs potentials.  Robustness implies various constraints on these potentials.
Section~\ref{sec:robustn-mark-kern} discusses the probabilistic behaviour of the whole system, including its inputs,
when the input variables are distributed to some fixed input distribution.  The set of all joint probability
distributions is found such that the system is $\Rcal$-robust for all input states with non-vanishing probability.  

Some of our results in Section~\ref{sec:robustness-and-CI} can also be derived from recent algebraic results
in~\cite{Rauh12:Binomial_edge_ideals} about
generalized binomial ideals.  These ideals generalize the binomial edge ideals of~\cite{HHHKR10:Binomial_Edge_Ideals}
and~\cite{Ohtani11:Ideals_of_some_2-minors}.  Similar ideals have recently been studied in the
paper~\cite{SwansonTaylor11:Minimial_Primes_of_CI_Ideals}, which discusses what we call $(n-1)$-robustness in
Section~\ref{sec:k-robustness}.  In this paper we give self-contained proofs that are also accessible to readers not
acquainted to the language of commutative algebra.  We comment on the relation to the algebraic results in
Remark~\ref{rem:algebra}.

\section{Robustness and canalyzing functions}
\label{sec:robust-canal}

We consider $n$ input nodes, denoted by $1,2,\dots,n$, and one output node, denoted by $0$.  For each $i=0,1,\dots,n$
the state of node $i$ is a discrete random variable $X_{i}$ taking values in the finite set $\Xcal_i$ of cardinality
$d_{i}$.
%  $[ d_i ] := \{1,\dots,d_i\}$.
The input state space is the set $\Xin=\Xcal_{1}\times\dots\times\Xcal_{n}$, and the joint state space is
$\Xtot=\Xcal_{0}\times\Xin$.  For any subset $S\subseteq\{0,\dots,n\}$ write $X_{S}$ for the random vector
$(X_{i})_{i\in S}$; then $X_{S}$ is a random variable with values in $\Xcal_{S}=\bigtimes_{i\in S}\Xcal_{i}$.  For
$S=[n]:=\{1,\dots,n\}$ we also write $\Xvin$ instead of~$X_{[n]}$.  For any $x\in\Xtot$, the \emph{restriction} of $x$
to a subset $S\subseteq\{0,\dots,n\}$ is the vector $x|_{S}\in\Xcal_{S}$ with $(x|_{S})_{i}=x_{i}$ for all $i\in S$.  In
contrast, the notation $x_{S}$ will refer to an arbitrary element of~$\Xcal_{S}$.

% We study two possible models for the computation of the output from the input: The first model is a
As a model for the computation of the output from the input, we use a stochastic map (Markov kernel) $\kappa$ from
$\Xin$ to $\Xcal_0$, that is, $\kappa$ is a function that assigns to each $x\in\Xin$ a probability distribution
$\kappa(x)$ for the output~$X_{0}$.  Such a stochastic map $\kappa$ can be represented by a matrix, with matrix elements
$\kappa(x;x_{0})$, $x\in\Xin, x_{0}\in\Xcal_{0}$,
% \begin{equation} \label{stochmap} 
%     \kappa:  \Xin \times \Xcal_0 \; \to \; [0,1], \qquad (x,x_{0}) \; \mapsto \; \kappa(x;x_{0}) \, ,
% \end{equation}
satisfying $\sum_{x_{0} \in \Xcal_{0}} \kappa(x;x_{0}) = 1$ for all~$x\in\Xin$.
For each $x\in\Xin$ the probability distribution $\kappa(x)$ models the behaviour of $X_{0}$ when the input variables
are in the state~$x$.  When the input is distributed according to some input distribution $\pin$, then the joint
distribution $p$ of input and output variables satisfies
\begin{equation*}
  p(X_{0}=x_{0},\Xvin=x) = \pin(\Xvin=x)\kappa(x;x_{0})\,.
\end{equation*}
If $\pin(\Xvin=x)>0$, then $\kappa(x)$ can be computed from the joint probability distribution $p$
and equals the conditional distribution of~$X_{0}$, given that $\Xvin=x$.

When a subset $S$ of the input nodes is knocked out and only the nodes in $R=[n]\setminus S$ remain, then the behaviour
of the system changes.  Without further assumptions, the post-knockout function is not determined by $\kappa$ and has to
be specified separately.  We model the post-knockout function by a further stochastic map $\kappa_R: \Xcal_{R} \times
\Xcal_0 \to [0,1]$.  A complete specification of the system is given by the family ${(\kappa_A)}_{A \subseteq [n]}$ of
all possible post-knockout functions, which we refer to as \emph{functional modalities}.  As a shorthand notation we
denote functional modalities by~$(\kappa_{A})$.  The stochastic map $\kappa$ itself, which describes the normal behaviour
of the system without knockouts, can be identified with~$\kappa_{[n]}$.

What does it mean for functional modalities to be robust?  Assume that the input is in state $x$, and that we knock out a set
$S$ of inputs.  Denoting the remaining set of inputs by~$R$, we say that $(\kappa_{A})$ is robust in $x$ % $ = (x|_R, x|_S)$
against knockout of $S$, if $\kappa(x) = \kappa_{R}(x|_{R})$, that is, if
\begin{equation} \label{invar}
  \kappa(x; x_0) = \kappa_R(x|_R ; x_0) \qquad \mbox{for all $x_0 \in \Xcal_0$} \, .
\end{equation}
Let $\Rcal$ be a collection of pairs $(R,x_{R})$, where $R\subseteq[n]$ and $x_{R}\in\Xcal_{R}$.  We call such a
collection a \emph{robustness specification} in the following.  % Let $\Scal\subseteq\Xin$.
We say that $(\kappa_{A})$ is \emph{$\Rcal$-robust} in a set $\Scal\subseteq\Xin$ if
\begin{equation} \label{invar-gen}
    \kappa(x) = \kappa_R(x|_R), \qquad \text{whenever $x\in\Scal$ and $(R,x|_{R})\in\Rcal$} \, .
    % \kappa(x; x_0) = \kappa_R(x|_R ; x_0) \qquad \text{for all $x_0 \in \Xcal_0$, whenever $x\in\Scal$ and $(R,x|_{R})\in\Rcal$} \, .
\end{equation}
The main example in this section will be the robustness structures
\begin{equation*}
  \Rcal_k := \Big\{ (R,x_{R}) \; : \; R \subseteq [n], |R| \geq k, x_{R}\in\Xcal_{R}\Big\}\,.
\end{equation*}

Equation~\eqref{invar} only compares the functional modality $\kappa_{R}$ after knockout with the stochastic map
$\kappa$ that describes the regular behaviour of the unperturbed system.  In particular, for $R\subsetneq
R'\subsetneq[n]$, the functional modality $\kappa_{R'}$ is in no way restricted by~\eqref{invar}.  Therefore, it may
happen that a system that is not robust against a knockout of a set $S'=[n]\setminus R'$ recovers its regular behaviour
if we knockout even more nodes.  However, this is not the typical situation.  Therefore, it is natural to assume that
the following holds: If $(R,x_{R})\in\Rcal$ and if $R\subsetneq R'\subsetneq[n]$, then $(R',x_{R'})\in\Rcal$ for all
$x_{R'}\in\Xcal_{R'}$ with $x_{R'}|_{R}=x_{R}$.  In this case we call the robustness specification $\Rcal$
\emph{coherent}.  For example the robustness structures $\Rcal_{k}$ are coherent.
The notion of coherence will not play an important role in the following, but it is interesting from a conceptual point
of view.  It is related to the notion of coherency as used e.g.~in~\cite{CabezonWynn12:Algebraic_Reliability_Review}.

% Now assume that $(\kappa_A)$ is $\Rcal$-robust in $\Scal$, and 
% consider a probability distribution $p$ with $\supp(p) = \Scal$.  Then we have for all $(R,x_R) \in \Rcal$ and $p(x_R) > 0$,
% \begin{eqnarray}
%    p(x_0 \, | \, x_R) & = & \sum_{x_S \atop (x_R, x_S) \in \supp (p)} p(x_S \, | \, x_R) \, p(x_0 \, | \, x_R, x_S)   \label{simpcalc}    \\ 
%                                 & = & \sum_{x_S \atop (x_R, x_S) \in \supp (p)} p(x_S \, | \, x_R) \, \kappa(x_R, x_S ; x_0)  \nonumber  \\ 
%                                 & = & \sum_{x_S \atop (x_R, x_S) \in \supp (p)} p(x_S \, | \, x_R) \, \kappa_R (x_R ; x_0) \nonumber \\ 
%                                 & = & \kappa_R (x_R ; x_0)   \sum_{x_S \atop (x_R, x_S) \in \supp (p)} p(x_S \, | \, x_R) \nonumber \\
%                                 & = & \kappa_R (x_R ; x_0) .   \nonumber
% \end{eqnarray}
% This simple calculation shows that for robust functional modalities $(\kappa_A)$ the largest functional modality $\kappa_{[n]}$ determines the smaller ones in the relevant points. 
% Given a robust $\kappa$, one can always associate a robust family $(\kappa_A)$ by setting $\kappa_R(x_R; x_0) := p(x_0 \, | \, x_R )$ whenever 
% $p(x_R) > 0$. This motivates the following definition: A stochastic map $\kappa$ is called robust, if there exist robust functional modalities $(\kappa_{A})$ with
% $\kappa=\kappa_{[n]}$. 
By definition, for robust functional modalities $(\kappa_A)$ the largest functional modality $\kappa_{[n]}$ determines
the smaller ones in the relevant points via~\eqref{invar-gen}.  This motivates the following definition: A stochastic
map $\kappa$ is called \emph{$\Rcal$-robust} in $\Scal$, if there exist functional modalities $(\kappa_{A})$ with
$\kappa=\kappa_{[n]}$ that are $\Rcal$-robust in $\Scal$.  More directly, $\kappa$ is $\Rcal$-robust in $\Scal$ if and
only if
\begin{equation*}
    \kappa(x) = \kappa(y), \qquad \text{whenever $x,y\in\Scal$, $x|_{R}=y|_{R}$ and $(R,x|_{R})\in\Rcal$} \, .
    % \kappa(x; x_0) = \kappa(y; x_0) \quad \text{for all $x_0 \in \Xcal_0$, whenever $x,y\in\Scal$, $x|_{R}=y|_{R}$ and $(R,x|_{R})\in\Rcal$} \, .
\end{equation*}
When studying robustness of a stochastic map $\kappa$ we may always assume that $\Rcal$ is coherent; for if
$x|_{R}=y|_{R}$ implies $\kappa(x)=\kappa(y)$, then $x|_{R'}=y|_{R'}$ also implies $\kappa(x)=\kappa(y)$, whenever
$R\subseteq R'\subseteq [n]$.

For any subset $R\subseteq[n]$ and $x_{R}\in\Xcal_{R}$ let
\[
     \mathcal{C}(R,x_R) \; := \; \big\{ x \in \Xin \; : \; x|_{R} = x_R \big\}.
\]
be the corresponding cylinder set.  Then $\kappa$ is $\Rcal$-robust in $\Scal$ if and only if
$\kappa(x)=\kappa(y)$ for all $x,y\in\Scal\cap\Ccal(R,x_{R})$ and
$(R,x_{R})\in\Rcal$.  In other words, the stochastic map $\kappa$ is constant on $\Scal\cap\Ccal(R,x_{R})$ for all
$(R,x_{R})\in\Rcal$.

%\begin{prop} %Let $\Rcal$ and $\Scal$, and let $p$ be a probability measure with $\supp(p) = \Scal$.   
%A family $(\kappa_R)$ of functional modalities is $\Rcal$-robust in $\Scal$ if and only if  
%    \begin{equation} \label{bedwahrsch}
%              \kappa_R(x_R,x_S ; x_0) \; := \; p(x_0 \, | \, x_R) \quad \mbox{for all $x_0 \in \Xcal_0$ and all $(x_R,x_S) \in \Scal$}.
%    \end{equation}
%%Furthermore, any family $(\kappa_R)$, $\kappa_{[n]} = \kappa $, that is $\Rcal$-robust in $\Scal$ satisfies (\ref{bedwahrsch}).   
%\end{prop}
%\begin{proof} ...
%\end{proof}

The following construction is useful to study robust functional modalities:
% In what follows we illustrate how to construct robust functional modalities.
Given a robustness specification $\Rcal$, define a graph $G_{\Rcal}$ on $\Xin$ by connecting two elements $x,y\in\Xin$
by an edge if there is $(R,x_R) \in \Rcal$ such that ${x |}_R = {y |}_R = x_R$.  Denote by $G_{\Rcal, \Scal}$ the
subgraph of $G_{\Rcal}$ induced by~$\Scal$.
\begin{ex}
  \label{ex:n-1-robustness-hypercube:graph}
  Assume that $\Xcal_{i}=\{0,1\}$ for $i=1,\dots,n$.  Then the input state space $\Xin=\{0,1\}^{n}$ can be identified
  with the vertices of an $n$-dimensional hypercube.  The graph $G_{\Rcal_{n-1}}$ is the edge graph of this hypercube
  (Fig.~\ref{fig:n-1-robustness-hypercube:graph}a)).  Cylinder sets correspond to faces of this hypercube.  If
  $R\subset[n]$ has cardinality~$n-1$, then the cylinder set $\Ccal(\Rcal,x_{R})$ is an edge, if $R$ has
  cardinality~$n-2$, then $\Ccal(\Rcal,x_{R})$ is a two-dimensional face.
  Fig.~\ref{fig:n-1-robustness-hypercube:graph}b) shows an induced subgraph of $G_{3}$ for $n=4$.  By comparison, the
  graph $G_{\Rcal_{n-2}}$ has additional edges corresponding to diagonals in the quadrangles of~$G_{\Rcal_{n-1}}$.  For
  example, the set of vertices marked black in Figure~\ref{fig:n-1-robustness-hypercube:graph}b) is connected
  in~$G_{\Rcal_{n-2}}$, but not in~$G_{\Rcal_{n-1}}$ (Fig.~\ref{fig:n-1-robustness-hypercube:graph}d)).
\end{ex}
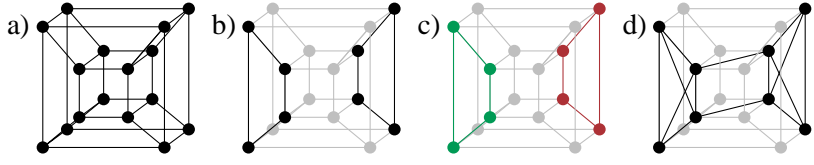
\begin{figure}
  \centering
  \begin{tikzpicture}[scale=0.8]
    \fourcubevert
    \draw (X13) -- (X15) -- (X14);
    \draw (X4) -- (X5);
    \draw (X7) -- (X6);
    \draw (X3) -- (X7);
    \draw (X1) -- (X5);
    \draw (X8) -- (X9);
    \draw (X11) -- (X10) -- (X8);
    \draw (X12) -- (X13);
    \draw (X11) -- (X15);
    \draw (X9) -- (X13);
    \draw (X2) -- (X3) -- (X11);
    \draw (X5) -- (X13);
    \draw (X7) -- (X15);
    \draw (X0) -- (X1);
    \draw (X0) -- (X4) -- (X6) -- (X2) -- (X0) -- (X8) -- (X12) -- (X14) -- (X10);
    \draw (X4) -- (X12);
    \draw (X6) -- (X14);
    \draw (X2) -- (X10);
    \draw (X1) -- (X3) -- (X11) -- (X9) -- cycle;
    \draw (X5) -- (X7);
    \foreach \i in {0,...,15} { \fill (X\i) circle (2.875pt); }
    \node at (-0.4,2) { a) };
  \end{tikzpicture}
  \begin{tikzpicture}[scale=0.8]
    \fourcubevert
    \draw (X5) -- (X7) -- (X15) -- (X13) -- (X5);
    \begin{scope}[lightgray]
      \draw (X15) -- (X14);
      \draw (X4) -- (X5);
      \draw (X7) -- (X6);
      \draw (X3) -- (X7);
      \draw (X1) -- (X5);
      \draw (X8) -- (X9);
      \draw (X11) -- (X10) -- (X8);
      \draw (X12) -- (X13);
      \draw (X11) -- (X15);
      \draw (X9) -- (X13);
      \draw (X2) -- (X3) -- (X11);
      \draw (X0) -- (X1);
      \draw (X0) -- (X4) -- (X6) -- (X2) -- (X0) -- (X8) -- (X12) -- (X14) -- (X10);
      \draw (X4) -- (X12);
      \draw (X6) -- (X14);
      \draw (X2) -- (X10);
      \draw (X1) -- (X3) -- (X11) -- (X9) -- cycle;
      \foreach \i in {1,3,4,6,9,11,12,14} { \fill (X\i) circle (2.875pt); }
    \end{scope}
    \foreach \i in {0,2,5,7,8,10,13,15} { \fill (X\i) circle (2.875pt); }
    \draw (X0) -- (X8) -- (X10) -- (X2) -- (X0);
    \node at (-0.4,2) { b) };
  \end{tikzpicture}
  \begin{tikzpicture}[scale=0.8]
    \fourcubevert
    \begin{scope}[Maroon]
      \draw (X5) -- (X7) -- (X15) -- (X13) -- (X5);
    \end{scope}
    \begin{scope}[lightgray]
      \draw (X15) -- (X14);
      \draw (X4) -- (X5);
      \draw (X7) -- (X6);
      \draw (X3) -- (X7);
      \draw (X1) -- (X5);
      \draw (X8) -- (X9);
      \draw (X11) -- (X10) -- (X8);
      \draw (X12) -- (X13);
      \draw (X11) -- (X15);
      \draw (X9) -- (X13);
      \draw (X2) -- (X3) -- (X11);
      \draw (X0) -- (X1);
      \draw (X0) -- (X4) -- (X6) -- (X2) -- (X0) -- (X8) -- (X12) -- (X14) -- (X10);
      \draw (X4) -- (X12);
      \draw (X6) -- (X14);
      \draw (X2) -- (X10);
      \draw (X1) -- (X3) -- (X11) -- (X9) -- cycle;
      \foreach \i in {1,3,4,6,9,11,12,14} { \fill (X\i) circle (2.875pt); }
    \end{scope}
    \begin{scope}[Maroon]
      \foreach \i in {5,7,13,15} { \fill (X\i) circle (2.875pt); }
    \end{scope}
    \begin{scope}[ForestGreen]
      \draw (X0) -- (X8) -- (X10) -- (X2) -- (X0);
      \foreach \i in {0,2,8,10} { \fill (X\i) circle (2.875pt); }
    \end{scope}
    \node at (-0.4,2) { c) };
  \end{tikzpicture}
  \begin{tikzpicture}[scale=0.8]
    \fourcubevert
    \draw (X5) -- (X7) -- (X15) -- (X13) -- (X5);
    \draw (X8) -- (X13) -- (X7);
    \draw (X10) -- (X15) -- (X5);
    \begin{scope}[lightgray]
      \draw (X15) -- (X14);
      \draw (X4) -- (X5);
      \draw (X7) -- (X6);
      \draw (X3) -- (X7);
      \draw (X1) -- (X5);
      \draw (X8) -- (X9);
      \draw (X11) -- (X10) -- (X8);
      \draw (X12) -- (X13);
      \draw (X11) -- (X15);
      \draw (X9) -- (X13);
      \draw (X2) -- (X3) -- (X11);
      \draw (X0) -- (X1);
      \draw (X0) -- (X4) -- (X6) -- (X2) -- (X0) -- (X8) -- (X12) -- (X14) -- (X10);
      \draw (X4) -- (X12);
      \draw (X6) -- (X14);
      \draw (X2) -- (X10);
      \draw (X1) -- (X3) -- (X11) -- (X9) -- cycle;
      \foreach \i in {1,3,4,6,9,11,12,14} { \fill (X\i) circle (2.875pt); }
    \end{scope}
    \draw (X0) -- (X8) -- (X10) -- (X2) -- (X0) -- (X10);
    \draw (X2) -- (X8);
    \foreach \i in {0,2,5,7,8,10,13,15} { \fill (X\i) circle (2.875pt); }
    \node at (-0.4,2) { d) };
  \end{tikzpicture}
  \caption{An illustration of Example~\ref{ex:n-1-robustness-hypercube:graph} with $n=4$. a) The
    graph~$G_{\Rcal_{3}}$.  b) An induced subgraph~$G_{\Rcal_{3},\Scal}$.  c) The connected components
    of~$G_{\Rcal_{3},\Scal}$.  In fact, in this example both connected components are cylinder sets.  d) The induced
    subgraph~$G_{\Rcal_{2},\Scal}$, which is connected.}
  \label{fig:n-1-robustness-hypercube:graph}
\end{figure}

\begin{prop} \label{verb}
  The following statements are equivalent for a stochastic map $\kappa$: \\
  {\bf (1)} $\kappa$ is $\Rcal$-robust in $\Scal$. \\
  {\bf (2)} $\kappa$ is constant on $\Scal\cap\Ccal(R,x_{R})$ for all $(R,x_{R})\in\Rcal$.\\
  {\bf (3)} $\kappa$ is constant on the connected components of  $G_{\Rcal, \Scal}$. \\
  {\bf (4)} For any probability distribution $\pin$ of $\Xvin$ with $\pin(\Scal)=1$ and for all
  $(R,x_{R})\in\Rcal$, the output $X_0$ is stochastically independent of $X_{[n] \setminus R}$ given $X_R = x_R$.
\end{prop}
\begin{proof} 
  The equivalence (1) $\Leftrightarrow$ (2) was already shown.
  \\
  (2) $\Leftrightarrow$ (3): Condition (2) says that $\kappa$ is constant along each edge of $G_{\Rcal,\Scal}$.  By
  iteration this implies (3).  In the other direction, the subgraph of $G_{\Rcal,\Scal}$ induced by
  $\Scal\cap\Ccal(R,x_{R})$ is connected for all $(R,x_{R})\in\Rcal$, and therefore (3) implies~(2).
  \\
  (2) $\Rightarrow$ (4):
%   Denote by $p$ the joint probability distribution defined via
%   \begin{equation*}
%     p(x_{0},x_{1},\dots,x_{n}) = \pin(x_{1},\dots,x_{n})\kappa(x_{1},\dots,x_{n}; x_{0}).
% %    p(X_{0}=x_{0},X_{1}=x_{1},\dots,X_{n}=x_{n})=\pin(X_{1}=x_{1},\dots,X_{n}=x_{n}).
%   \end{equation*}
  % The calculation
  % \begin{eqnarray}
  %   p(x_0 \, | \, x_R) & = & \sum_{x_S \atop (x_R, x_S) \in \supp (p)} p(x_S \, | \, x_R) \, p(x_0 \, | \, x_R, x_S)   \label{simpcalc}    \\ 
  %                               & = & \sum_{x_S \atop (x_R, x_S) \in \supp (p)} p(x_S \, | \, x_R) \, \kappa(x_R, x_S ; x_0)  \nonumber  \\ 
  %                               & = & \sum_{x_S \atop (x_R, x_S) \in \supp (p)} p(x_S \, | \, x_R) \, \kappa_R (x_R ; x_0) \nonumber \\ 
  %                               & = & \kappa_R (x_R ; x_0)   \sum_{x_S \atop (x_R, x_S) \in \supp (p)} p(x_S \, | \, x_R) \nonumber \\
  %                               & = & \kappa_R (x_R ; x_0) .   \nonumber
  % \end{eqnarray} % (\ref{simpcalc})
  % implies $p(x_0 \, | \, x_R) = p(x_0 \, | \, x_R, x_S)$ whenever $(R,x_R) \in \Rcal$ and $(x_R, x_S) \in \Scal$. \\
  For any $x\in\Xin$ with $\pin(x)>0$, the conditional distribution of the output given the input satisfies
%  \begin{equation*}
%    p(x_{0}|x) :=
  $p(X_{0}=x_{0}\,|\,\Xvin=x) = \kappa(x;x_{0})$.
%  \end{equation*}
  By (2), $\kappa(x;x_{0})$ is constant on $\Ccal(\Rcal,x|_{R})\cap\Scal$.  Hence the conditional distribution does not
  depend on $X_{[n]\setminus R}$, and so $p(X_{0}=x_{0}\,|\,\Xvin=x)=p(X_{0}=x_{0}\,|\,X_{R}=x|_{R})$.
  \\
  (4) $\Rightarrow$ (2): Let $\pin$ be the uniform distribution on~$\Scal$ (or any other probability distribution with
  support~$\Scal$), and fix $(\Rcal,x_{R})\in\Rcal$.  By assumption, for any $x\in\Scal$ with $x|_{R}=x_{R}$, the
  conditional distribution $p(X_{0}=x_{0}\,|\,\Xvin=x) = \kappa(x;x_{0})$ does not depend on $x|_{[n]\setminus
    S}$.  Therefore, $\kappa(x)$ is constant on $\Scal\cap\Ccal(\Rcal; x_{R})$.
  % Simply set $\kappa_R(x_R; x_0) := p(x_0 \, | \, x_R)$ whenever $p(x_R) > 0$.
\end{proof}

The choice of the set $\Scal$ is important: On one hand $\Scal$ should be large, because otherwise the notion of
robustness is very weak.  However, if $\Scal$ is too large, then the equations~\eqref{invar} imply that the output
$X_{0}$ is (unconditionally) independent of all inputs.
% The theory developed in Sections~\ref{sec:robustness-and-CI} to~\ref{sec:k-robustness} discusses the constraints on
% conditionals imposed by the choice of~$\Scal$.  In particular, Section~\ref{sec:robust_functions} gives bounds on the
% strength of the interaction between the input nodes and the output node for given $\Rcal$ and $\Scal$.
Proposition~\ref{verb} gives a hint how to choose the set~$\Scal$: The goal is to have as many connected components as
possible in~$G_{\Rcal,\Scal}$.
% We will derive bounds on the number of connected components in Section~\ref{sec:robust_functions}.
This motivates the following definition:
\begin{defi}
  \label{def:robustness-structure}
  For any subset $\Scal\subseteq\Xin$, the set of connected components of $G_{\Rcal,\Scal}$ is called an
  \emph{$\Rcal$-robustness structure}.
\end{defi}

Let $\Bbf$ be an $\Rcal$-robustness structure, and let $\Scal=\cup\Bbf$.  Let $f_{\Bbf}:\Scal\to\Bbf$ be the map that
maps each $x\in\Scal$ to the corresponding block of $\Bbf$ containing~$x$.  Any stochastic map $\kappa$ that is
$\Rcal$-robust on $\Scal$
factorizes through~$f_{\Bbf}$, in the sense that there is a stochastic map $\kappa'$ that maps each block in $\Bbf$ to a
probability distribution on~$\Xcal_{0}$ and that satisfies~$\kappa=\kappa'\circ f_{\Bbf}$.  Conversely, any stochastic
map $\kappa$ that factorizes through $f_{\Bbf}$ is $\Rcal$-robust.

To any joint probability distribution $\pin$ on $\Xin$ with $p(\Xvin\in\Scal)=1$ we can associate a random variable
$B=f_{\Bbf}(X_{1},\dots,X_{n})$.  If $\kappa$ is $\Rcal$-robust on $\Scal$, then $X_{0}$ is independent of
$X_{1},\dots,X_{n}$ given~$B$.
% $\CI{X_{0}}{\{X_{1},\dots,X_{n}\}}[B]$.
Note that the random variable $B$ is only defined on $\cup\Bbf$, which is a set of measure one with respect to~$\pin$.
The situation is illustrated by the following graph:
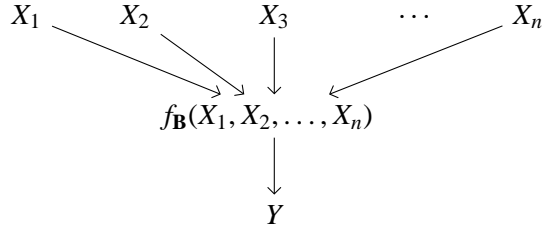
\begin{center}
  \begin{tikzpicture}
    % \foreach \i in {1,2,3} { \path (\i-3,0) coordinate (X\i); \node at (X\i) { $X_{\i}$ }; }
    % \draw (X1) -- (X2);
    \matrix (m) [matrix of math nodes, text height=1.5ex, text depth=0.25ex, column sep=2em, row sep=2em]
    { |(X1)| X_1 & |(X2)| X_2 & |(X3)| X_3 & \cdots & |(X4)| X_n \\
       & & |(B)| \mathllap{f_{\Bbf}(X_{1}},X_{2},\dots\mathrlap{,X_{n})}  & &  \\
       & & |(Y)| Y & &  \\
    }; \foreach \i in {1,2,3,4} { \path[-angle 90] (X\i) edge (B); }
    \path[-angle 90] (B) edge (Y);
  \end{tikzpicture}
\end{center}
% \begin{equation*}
%       \xymatrix{
%          && X_{0} && \\
%          && \mathllap{f_{\Bbf}(X_{1}},X_{2},\dots\mathrlap{,X_{n})}\ar[u] && \\
%         X_{1}\ar[urr] & X_{2}\ar[ur] & X_{3}\ar[u] & \cdots & X_{n}\ar[ull]\\
%      }
% \end{equation*}

When the robustness specification~$\Rcal$ is fixed, how much freedom is left to choose a robust stochastic map~$\kappa$?
More precisely, % how large can the image of $f$ be?  Equivalently,
how many components can an $\Rcal$-robustness structure $\Bbf$ have?

\begin{lemma}
  \label{lem:number-bound}
  Let $\Bbf$ be a robustness structure of the robustness specification $\Rcal$.  Let $R\subseteq[n]$, $S=[n]\setminus R$
  and $\Ycal_{R}=\{x_{R}\in\Xcal_{R}:(R,x_{R})\in\Rcal\}$.  Then
  \begin{equation*}
    |\Bbf| \le |\Ycal_{R}| + |\Xcal_{R}\setminus\Ycal_{R}|\cdot|\Xcal_{S}|.
  \end{equation*}
  % The cardinality of any $\Rcal$-robustness structure is bounded from above by
  % \begin{equation*}
  %   \min\left\{\prod_{i\in R}d_{i}: (R,x_{R})\in \Rcal\text{ for all }x_{R}\in \Xcal_{R} \right\}.
  % \end{equation*}
\end{lemma}
\begin{proof}
%  Denote by $m$ the right hand side of the inequality.
  The set $\Scal$ is the disjoint union of the $|\Ycal_{R}|$ sets $\Ccal(R,x_{R})\cap\Scal$ for $x_{R}\in\Ycal_{R}$ and
  the $|\Xcal_{R}\setminus\Ycal_{R}|\cdot|\Xcal_{S}|$ singletons $\{x\}\subset\Scal$ with $x|_{R}\notin\Ycal_{R}$.  Each
  of these sets induces a connected subgraph of~$G_{\Rcal}$.  The statement now follows from Proposition~\ref{verb}.
  % Suppose without loss of generality that $(\{1,\dots,r\},x)\in\Rcal$ for all $x\in\Xcal_{[r]}$.
  % The image of $f$ cannot be larger than $d_{1}\cdots d_{r}$, since if we knock out all $X_{i}$ for $i>r$, then we can
  % only determine $d_{1}\cdots d_{r}$ states.
\end{proof}

\begin{ex}
  Suppose that $\Scal=\Xin$.  This means that any $\Rcal$-robustness structure $\Bbf$ satisfies $\cup\Bbf = \Xin$.  If
  $G_{R}$ is connected, then $\Bbf$ has just a single block.  In this case the bound of Lemma~\ref{lem:number-bound} is
  usually not tight.  On the other hand, the bound is tight if $\Rcal=\{(R,x_{R}) : x_{R}\in\Xcal_{R}\}$.
%
  % We first consider the case that $G_{\Rcal}$ is connected.  This is fulfilled, for example, if for any $k\in[n]$ there
  % exists $R\subseteq[n]$ such that $x\notin R$ and $(R,x_{R})\in\Rcal$ for all $x_{R}\in\Xcal_{R}$.  In this case
  % $G_{\Rcal}$ equals the complete graph, and any induced subgraph is connected.
%
  % Assume that $(R,x_{R})\in\Rcal$ implies $(R,y_{R})\in\Rcal$ for all $y_{R}\in\Xcal_{R}$.  If $G_{\Rcal}$ is not connected, then
  % some input variables may never be knocked out.  Let $T$ be the set of these input variables.  For every fixed value of
  % $X_{T}$ the function $f$ must be constant.  This means that any $\Rcal$-robustness structure % $f$
  % can have $\prod_{i\in T}d_{i}$ different blocks.
\end{ex}

\begin{rem}[Relation to coding theory]
  \label{rem:coding}
  Assume that all $d_{i}$ are equal.
  We can interpret $\Xin$ as the set of words of length $n$ over the alphabet $[d_{1}]$.
  Consider the uniform case $\Rcal=\Rcal_{k}$.  Then the task is to find a collection of subsets such that any two
  different subsets have Hamming distance at least $n-k+1$.  A related problem appears in coding theory: A code is a
  subset $\Ycal$ of $\Xin$ and corresponds to the case that each element of $\Bbf$ is a singleton.  If distinct elements
  of the code have Hamming distance at least $n-k+1$, then a message can be reliably decoded even if only $k$ letters
  are transmitted.  If all letters are transmitted, but up to $k$ letters may contain an error, then this error may at
  least be detected; hence such codes are called \emph{error detecting codes}. % correctly.
  In this setting, the function $f_{\Bbf}$ can be interpreted as the
  decoding operation.  The problem of finding a largest possible code such that all code words have a fixed minimum
  distance is also known as the \emph{sphere packing problem}.  The maximal size $A_{d_{1}}(n,n-k+1)$ of such a code is
  unknown in general.
\end{rem}
% \begin{rem}[Relation to the theory of neutral sets]
%   This notion of robustness is related to the notion of \emph{neutral sets}, as used in the study of evolution.
%   Consider a mapping $f:\Gcal\to\Pcal$ from the space $\Gcal$ of genotypes to the space $\Pcal$ of phenotypes.  
% % Usually, $\Gcal$ is a product space, $\Gcal=\Gcal_{1}\times\dots\times\Gcal_{n}$, corresponding to different alleles for $n$
% % genes.
%   $\Gcal$ can be made a graph by adding an edge between two genotypes $g,g'$ if there is a `simple' mutation from $g$
%   to~$g'$.  A neutral set is a set of genotypes that is mapped to the same phenotype.  Neutrality is an important
%   concept when studying robustness and evolvability.  It is related to our notion of robustness.  In fact, most of our
%   results also apply to the study of neutral graphs by replacing our graph $G_{\Rcal}$ with the graph~$\Gcal$.
% %
  
%   Usually, in the study of neutral sets, the mapping $f$ is given from the beginning, and the goal is to find the
%   neutral sets.  In our work, we take a different perspective: Fixing a robustness structure and a set~$\Scal$
%   corresponds to fixing the neutral
%   % The graph $G_{\Rcal}$ is slightly more regular: Let $(x,y)$ be an edge in $G_{\Rcal}$, and assume that
%   % $x|_{R}=y|_{R}$.  If $z\in\Xcal$ satisfies $z|_{R}=x|_{R}$, then $(x,z)$ and $(y,z)$ are also edges in~$G_{\Rcal}$.
%   % However, this regularity condition is not used in our analysis.
% %  NeverthFor simplicity we omit the details.
% \end{rem}

\section{Canalyzing functions}
\label{sec:canalyzing-functions}

Our notion of $\Rcal$-robust functional modalities naturally generalizes and is motivated by
canalyzing~\cite{Kauffman93:Origins_of_Order} and nested canalyzing
functions~\cite{KauffmanPetersonSamuelssonTroein04:Genetic_networks_with_canalyzing_Boolean_rules}.  Let $f: \Xin \to
\Xcal_0$ be a function, also called (deterministic) map. Such a map can be considered as a special case of a stochastic
map by identifying $f$ with % as described above (see (\ref{stochmap})):
  \begin{equation*} % \label{detasstoch}
      \kappa^{f}(x; x_{0}) \; := \; 
      \left\{
         \begin{array}{c@{,\quad}l}
           1 & \mbox{if $f(x) = x_{0}$} \\
           0 & \mbox{otherwise}
         \end{array}
      \right. .
  \end{equation*}
% Canalyzation is a concept related to robustness:
We say that $f$ is \emph{$(R,x_R)$-canalyz\-ing}, if the value
of $f$ does not depend on the input variables $X_{[n] \setminus R}$ given 
that the input variables $X_R$ are in state $x_R$. In other words, 
an $(R,x_R)$-canalyzing function is assumed to be constant
on the corresponding cylinder set:
\begin{equation*} % \label{constonm}
     x,x' \in \mathcal{C}(R,x_R) \quad \Rightarrow \quad f(x) = f(x').
\end{equation*}
Given a robustness specification $\Rcal$, we say that a function $f$ is 
$\Rcal$-{\em canalyzing\/} if it is $(R,x_R)$-canalyzing for all 
$(R,x_R) \in \Rcal$. 
Clearly, the set of $\Rcal$-canalyzing functions strongly depends on 
$\Rcal$. On one hand, any function is $\Rcal$-canalyzing with respect to
\[
    \Rcal \; = \; \big\{ ([n], x) \; : \; x \in \Xin \big\}.   
\]  
On the other hand, for two different elements $i, j \in [n]$, and   
\[
   \Rcal = \big( \{i\} \times \Xcal_i \big) \cup \big(  \{j\} \times \Xcal_j \big),
\]
any $\Rcal$-canalyzing function is constant. Note that  
constant functions are $\Rcal$-canalyzing for any $\Rcal$.

The following statement directly follows from Proposition~\ref{verb}:
\begin{prop}
  A function $f:\Xin\to\Xcal_{0}$ is $\Rcal$-canalyzing if and only if $\kappa^{f}$ is $\Rcal$-robust in $\Scal=\Xin$.
% Any $\Rcal$-canalyzing function $f$ is $\Rcal$-robust in $\Scal = \Xin$.
\end{prop}

% \begin{proof} One easily verifies that a function is constant on the cylinder sets $\mathcal{C}(R,x_R)$ if and only if it is constant on the 
% connected components of $G_{\Rcal, \Scal}$ with $\Scal = \Xin$. The statement then follows from Proposition \ref{verb}.
% \end{proof}

%Condition (\ref{constonm}) can be quite restrictive, depending on the structure 
%of $\Rcal$. In order to see this, consider two different input nodes $i$ and 
%$j$, and 
%\[
%    \Rcal \; = \; \big\{ (\{i\}, x_i) \; : \; x_i \in \Xcal_i)\big\} \cup 
%                        \big\{ (\{j\}, x_j) \; : \; x_j \in \Xcal_j)\big\}. 
%\] 
%This already implies that any $\Rcal$-canalyzing function must be constant. 
Particular cases of canalyzing functions have been studied in the context of robustness:
%~\cite{Kauffman93:Origins_of_Order}, ~\cite{JarrahRaposaLaubenbacher07:Nested_Canalyzing_and_other_Functions}. 

\begin{ex}
 \emph{(1) Canalyzing functions.}   A function $f$ with domain $\Xin$ is \emph{canalyzing} 
 in the sense of~\cite{Kauffman93:Origins_of_Order},
 if there exist an input node $k\in[n]$, an input value $a \in\Xcal_{k}$, and an output value 
  $b \in \Xcal_0$ such that the value of $f$ is independent of 
  $x_{[n]\setminus\{k\}}$, given that $x|_{k}=a$.  In other words, $f(x)=f(y)= b$ whenever 
  $x|_{k}=y|_{k}=a$. A canalyzing function is $\Rcal$-canalyzing with
  \[
%  \Rcal=\{(\{k\},a)\}\,.
      \Rcal := \big\{ (R,x_{R}) \; : \; R \subseteq [n], \; k \in R, \; x_{R} \in
      \Xcal_R,\; {x_{R}|}_k = a \big\}\,.
  \]  
  
  \emph{(2) Nested canalyzing functions} % These functions
  have been studied in~\cite{KauffmanPetersonSamuelssonTroein04:Genetic_networks_with_canalyzing_Boolean_rules}.  A
  function $f$ is \emph{nested canalyzing} in the variable order $X_{1},\dots,X_{n}$ with canalyzing input values
  $a_{1}\in\Xcal_{1}$, \dots, $a_{n}\in\Xcal_{n}$ and canalyzed output values $b_{1},\dots,b_{n}$ if $f$ satisfies
  $f(x)=b_{k}$ for all $x\in\Xcal$ satisfying $x|_{k}=a_{k}$ and $x|_{i}\neq a_{i}$ for all $i<k$.  Let $\Rcal :=
  \biguplus_{k = 1}^n \Rcal^{(k)}$, where
  \begin{equation*}
     \Rcal^{(k)} \; := \; 
         \Big\{
           (R,x_{R}) \; : \; [k] \subseteq R,\; x_{R}|_{1} \neq a_1, \dots, x_{R}|_{k-1}
           \neq a_{k-1}, x_{R}|_k = a_k  
         \Big\},
      % \Rcal_{k} \; := \; \Big\{([k],(x_{1},\dots,x_{k-1},a_{k})) \;:\; x_{1}\neq a_{1},\dots,x_{k-1}\neq a_{k-1} \Big\},
%         \quad \Rcal := \biguplus_{k = 1}^n \Rcal^{(k)}
         \,.
  \end{equation*}
  It is easy to see that $f$ is a nested canalyzing function if and only if it is $\Rcal$-canalyzing.

  The set of Boolean nested canalyzing functions has been described algebraically
  in~\cite{JarrahLaubenbacher07:Toric_Variety_of_Nested_Canalyzing_Functions} as a variety over the finite field
  $\Fb_{2}$.  Here, we use a different viewpoint, which allows to study not only deterministic functions, but also
  stochastic functions. 
\end{ex}

\section{Robustness and Gibbs representation}
\label{sec:robustn-mark-kern}

Let $(\kappa_A)$ be a collection of functional modalities, as defined in Section~\ref{sec:robust-canal}.  Instead of
providing a list of all functional modes $\kappa_A$, one can describe them in more mechanistic terms.  To illustrate
this, we first consider an example from the field of neural networks: Assume that the output node receives an input $x =
(x_1,\dots,x_n) \in \{-1,+1\}^n$ and generates the output % $x_{0}\in\{-1,+1\}$ with probability
$+1$ with probability
\begin{equation*} % \label{neuron}
    \kappa(x_1,\dots,x_n; +1) \; := \; \frac{1}{1 + e^{-(\sum_{i = 1}^n w_i \, x_i - \eta)}} \, .
\end{equation*}
For an arbitrary output $x_0$ this implies
\begin{equation} \label{neuron2} % \frac{1}{2}
  \kappa(x_1,\dots,x_n; x_0) \; := \;
  \frac{ e^{\frac{1}{2}(\sum_{i = 1}^n w_i \, x_i \, - \eta) x_0 }}{e^{\frac{1}{2}(\sum_{i = 1}^n w_i \, x_i  - \eta)\cdot (-1)} + e^{\frac{1}{2} (\sum_{i = 1}^n w_i \, x_i  - \eta)\cdot ( + 1 )}} \, .
\end{equation}
% This representation of the stochastic map $\kappa$ has a structure that allows to infer the function after a knockout of
The structure of this representation of the stochastic map $\kappa$ already suggests what the function should be after a
knockout of a set $S$ of input nodes: Simply remove the contribution of all the nodes in~$S$.  The post-knockout
function is then given by
\begin{equation}\label{eq:neuron-funcmod}
%     \kappa_R(x_R; +1) \; := \; \frac{1}{1 + e^{-\sum_{i \in R} w_i \, x_i}},
     \kappa_R(x_R; +1) \; := \; \frac{e^{\frac{1}{2}(\sum_{i \in R} w_i \, x_i - \eta)\,x_{0}}}{e^{+\frac{1}{2}(\sum_{i \in R} w_i \, x_i - \eta)} + e^{-\frac{1}{2}(\sum_{i \in R} w_i \, x_i - \eta)}},
\end{equation}
where $R=[n]\setminus S$.  % This inference of the post-knockout function is
These post-knockout functional modalities are based on the decomposition of the sum that appears in~\eqref{neuron2}.
%~\eqref{neuron}.

More generally, we consider the following model of $(\kappa_A)$:
\begin{eqnarray} \label{intpot}
    \kappa_A(x_A; x_0) \; = \; \frac{e^{\sum_{B \subseteq A} \phi_B(x_A|_{B},x_0)}}{\sum_{x_0'} e^{ \sum_{B \subseteq A} \phi_B(x_A|_{B},x_0')}},
\end{eqnarray}
where the $\phi_{B}$ are functions on $\Xcal_{B} \times \Xcal_0$.  Such a sum decomposition of $\kappa$ is
referred to as a \emph{Gibbs representation} of $\kappa$ and contains more information than $\kappa$ itself.  Clearly,
each $\kappa_A$ is strictly positive.  Using the M\"obius inversion, it is easy to see that each strictly positive
family $(\kappa_A)$ has a representation of the form~\eqref{intpot} with
%.  To this end, we simply set
\begin{equation}
  \label{eq:Moebiuspotentials}
  \phi_A (x_A, x_0) \; := \; \sum_{C \subseteq A} (-1)^{| A \setminus C |} \ln \kappa_C(x_A|_{C} ; x_0 ) \, .
\end{equation}
Note that this representation is not unique: If an arbitrary function of $x_{A}$ (that does not depend on $x_{0}$) is
added to the function $\phi_{A}$, then the function~$\kappa_{A}$, defined via~\eqref{intpot}, does not change.

A single robustness constraint has the following consequences for the $\phi_{A}$.
\begin{prop}
  \label{prop:robkernelcond}
  Let $S\subseteq[n]$ and $R=[n]\setminus S$, and let $(\kappa_A)$ be strictly positive functional modalities with
  Gibbs potentials $(\phi_{A})$.  Then $(\kappa_A)$ is robust in $x$ against knockout of $S$ if and only if
  $\sum_{\substack{B\subseteq[n], B\not\subseteq R}}\phi_{B}(x|_{B},x_{0})$ does not depend on $x_{0}$.
\end{prop}
\begin{proof}
%  If $(\kappa_{A})$ is robust, then
% define Gibbs potentials via~\eqref{eq:Moebiuspotentials}.
  Denote by $\tilde\phi_{A}$ the potentials defined via~\eqref{eq:Moebiuspotentials}.  Then~\eqref{invar} is equivalent to
  \begin{equation*}
    \sum_{B\subseteq[n]}\tilde\phi_{B}(x|_{B},x_{0})
    =
    \sum_{B\subseteq R}\tilde\phi_{B}(x|_{B},x_{0})
%    \\
    \quad\Longleftrightarrow\quad
    \sum_{\substack{B\subseteq[n]\\B\not\subseteq R}}\tilde\phi_{B}(x|_{B},x_{0})
    = 0.
  \end{equation*}
  The statement follows from the fact that $\phi_{B}(x|_{B};x_{0}) - \tilde\phi_{B}(x|_{B};x_{0})$ is independent
  of~$x_{0}$ (for fixed $x$).
\end{proof}
\begin{ex}
  Consider $n=2$ binary inputs, $\Xcal_{1}=\Xcal_{2}=\{0,1\}$, and let $\Scal=\{(0,0),(1,1)\}$.  Then $1$-robustness
  on $\Scal$ means
  \begin{equation*}
    \kappa_{\{1\}}(x_{1}; x_{0}) = \kappa_{\{1,2\}}(x_{1},x_{2}; x_{0}) = \kappa_{\{2\}}(x_{2}; x_{0})
  \end{equation*}
  for all $x_{0}$ whenever $x_{1}=x_{2}$.  By Proposition~\ref{prop:robkernelcond} this translates into the conditions
  \begin{equation}
    \label{eq:ex3:potcond}
    \phi_{\{1,2\}}(x_{1},x_{2}; x_{0}) + \phi_{\{1\}}(x_{1}; x_{0})
%    &= 0, \\
    = 0 =
    \phi_{\{1,2\}}(x_{1},x_{2}; x_{0}) + \phi_{\{2\}}(x_{2}; x_{0})
%    &= 0
  \end{equation}
  for all $x_{0}$ whenever $x_{1}=x_{2}$
  for the potentials $(\phi_{A})$ defined via~\eqref{eq:Moebiuspotentials}.  This means: Assuming that
  $(\kappa_{A})$ is $1$-robust, it suffices to specify the four functions
  \begin{align*}
    \phi_{\emptyset}(x_{0}),   \;
    \phi_{\{1\}}(x_{1}; x_{0}),\;
    \phi_{\{1,2\}}(0,1; x_{0}),\;
    \phi_{\{1,2\}}(1,0; x_{0}).
  \end{align*}
  The remaining potentials can be deduced from~\eqref{eq:ex3:potcond}.  If only the values of $(\kappa_{A})$ for $x\in\Scal$ are needed, then it suffices to specify
%  \begin{equation*}
$    \phi_{\emptyset}(x_{0})
    \text{ and }
    \phi_{\{1\}}(x_{1}; x_{0})$.
%  \end{equation*}
\end{ex}

Does $\Rcal$-robustness in $x$ imply any structural constraints on $(\kappa_A)$?
%
% In the following we study this question in the special case $\Rcal = \Rcal_k := \{ (R,x_{R}) \; : \; R \subseteq [n], |
% R | \geq k, x_{R}\in\Xcal_{R}\}$.  For simplicity, we replace any prefix or subscript $\Rcal_{k}$ by $k$.  We shortly
% discuss how our results generalize to arbitrary $\Rcal$ in Remark~\ref{rem:general-structure}.
%
If $(\kappa_{A})$ is $\Rcal$-robust in $x$ for all $x$ belonging to a set $\Scal$, then the corresponding conditions imposed by
Proposition~\ref{prop:robkernelcond} depend on $\Scal$.  In this section, we are interested in conditions that are
independent of $\Scal$.  Such conditions allow to define sets of functional modalities that contain all $\Rcal$-robust
functional modalities for all possible sets $\Scal$.  If $\Scal$ (which will be the support of the input distribution in
Section~\ref{sec:robustness-and-CI}) is unknown from the beginning, then the system can choose its policy within such a
restricted set of functional modalities.  To find results that are independent of $\Scal$, our trick is to find a set
$\tilde M_{\Rcal}$ of functional modalities % $(\tilde\kappa_{A})$
such that $(\kappa_{A})$ % and $(\tilde\kappa_{A})$ agree on~$\Scal$.
can be approximated on $\Scal$ by functional modalities in~$\tilde M_{\Rcal}$.  The approximation will be independent
of~$\Scal$. % , but it will be valid only on~$\Scal$.

We first consider the special case $\Rcal=\Rcal_k := \{ (R,x_{R}) \,:\, R\subseteq [n], |R|\geq k, x_{R}\in\Xcal_{R}\}$.
For simplicity, we replace any prefix or subscript $\Rcal_{k}$ by $k$.
Denote by $M_{k+1}$ the set of all functional modalities $(\kappa_{A})$ such that there exist potentials $\phi_{A}$ of the
form
\begin{equation*}
  \phi_{A}(x_{A};x_{0}) = \sum_{\substack{B\subseteq A \\ |B|< k+1}}\alpha_{A,B}\Psi_{B}(x_{A}|_{B};x_{0}),
\end{equation*}
where $\alpha_{A,B}\in\Rb$ and $\Psi_{B}$ is an arbitrary function $\Rb^{\Xcal_{B}\times\Xcal_{0}}\to\Rb$.
The set $M_{k+1}$ is called the \emph{family of $(k+1)$-interaction functional modalities}.  Note that the functions
$\Psi_{B}$ do not depend on~$A$.  This ensures a certain interdependence among the functional modalities~$\kappa_{A}$.  The
name ``$(k+1)$-interaction'' comes from the fact that each potential $\Psi_{B}$ depends on the $k$ (or less) variables
in $B$ plus the output variable~$X_{0}$.  Since $M_{k+1}$ only contains strictly positive functional modalities, we are also
interested in the \emph{closure} of $M_{k+1}$ with respect to the usual topology on the space of matrices, considered
as elements of a finite-dimensional real vector space.
\begin{ex}
  \label{ex:linear-threshold}
  The functional modalities~\eqref{eq:neuron-funcmod}, derived from the classical model~\eqref{neuron2} of a neural
  network, belong to~$M_{2}$.  To illustrate the difference between $M_{2}$ and its closure, consider the functional
  modalities $(\kappa_{A})$ with
  \begin{equation*}
    \kappa_{A}(x_1,\dots,x_n; x_0) \; := \; \frac{ e^{\frac{\beta}{2}(\sum_{i\in A} w_i \, x_i\,-\,\eta)\,x_0}}{e^{-\frac{\beta}{2}(\sum_{i\in A} w_i \, x_i\,-\,\eta)}
      + e^{+\frac{\beta}{2}(\sum_{i\in A}^n w_i \, x_i\,-\,\eta)}} \, .
  \end{equation*}
  If $w_{1},\dots,w_{n}$ and $\eta$ are fixed and $\beta\to\infty$, then
  \begin{equation}
    \label{ex:signum-modalities}
    \kappa_{A}(x_1,\dots,x_n; +1)
    \to \theta(\sum_{i\in A} w_{i}x_{i} - \eta),
  \end{equation}
  where
  \begin{equation*}
    \theta(x) =
    \begin{cases}
      +1, &\text{ if } x > 0,\\
      \frac12, &\text{ if } x = 0,\\
      0, &\text{ if } x < 0.
    \end{cases}
    % \sgn(x) =
    % \begin{cases}
    %   +1, &\text{ if } x\ge 0,\\
    %   -1, &\text{ if } x\le 0.
    % \end{cases}
  \end{equation*}
  The functional modalities~\eqref{ex:signum-modalities} are deterministic limits of the probabilistic
  model~\eqref{neuron2}, called \emph{linear threshold functions}.  They lie in the closure
  of~$M_{2}$, but not in $M_{2}$ itself.

  Linear threshold functions are widely used as elementary building blocks in network dynamics, for example to build
  simple models of neural networks, metabolic networks or gene-regulation networks.  Robustness against knockouts of
  such networks has been studied in~\cite{BBROK10:Robustness_of_Boolean_Dynamics}, exploring the example of the yeast
  cell cycle.
\end{ex}

Let $\tilde M_{k+1}$ be the set of strictly positive functional modalities $(\kappa_{A})$ such that
\begin{multline}
  \label{eq:geom-mean-k}
  \kappa_{C}(x_{C};x_{0})
  = \frac{1}{Z_{C,x_{C}}} \exp\left(\sum_{\substack{B\subseteq C\\|B|=k}}\frac{1}{\binom{|C|}{k}}\ln(\kappa_{B}(x_{C}|_{B};x_{0}))\right)
  \\
  = \frac{1}{Z_{C,x_{C}}} \left(\prod_{\substack{B\subseteq C\\|B|=k}}\kappa_{B}(x_{C}|_{B};x_{0})\right)^{1/{\binom{|C|}{k}}}
% {\frac{1}{\binom{|C|}{k}}}
\end{multline}
for all $C\subseteq[n]$ with~$|C|>k$, where $Z_{C,x_{C}}$ is a normalization constant that ensures that $\kappa_{C}(x_{C})$ is
a probability distribution.
% $\tilde M_{k+1}$ consists of those functional modalities $(\kappa_{A})$, where the stochastic maps $\kappa_{C}$ for
% $|C|\le k$ are arbitrary, and where the stochastic maps $\kappa_{C}$ for $|C|>k$ are obtained by normalizing the
% geometric mean of the stochastic maps $\kappa_{B}$ for $B\subseteq C$ and~$|B|=k$.
%
% Again, $\tilde M_{k+1}$ only contains strictly positive kernels, and hence its closure will be important in the
% following.
Note that equations~\eqref{eq:geom-mean-k} can be used to parametrize the set $\tilde M_{k+1}$: The stochastic
maps $\kappa_{A}$ with $|A|\le k$ can be chosen arbitrarily, while all other stochastic maps $\kappa_{C}$ with $|C|>k$
can be computed by normalizing the geometric mean of the stochastic maps $\kappa_{B}$ for $B\subseteq C$ and~$|B|=k$.

\begin{lemma}
  \label{lem:Ktilde-und-K-k+1}
  $\tilde M_{k+1}$ is a subset of~$M_{k+1}$.  It consists of those functional modalities $(\kappa_{A})$ where the
  coefficients $\alpha_{A,B}$ additionally satisfy
  \begin{equation*}
    (-1)^{|A|}\alpha_{A,B} = (-1)^{|A'|}\alpha_{A',B},\qquad\text{ whenever $B\subseteq A\cap A'$ and $|B|< k$},
  \end{equation*}
  and
  \begin{equation*}
  % \sum_{R \subseteq A'  \setminus C } \frac{(-1)^{| A' | - |R| - k}}{ \binom{|R| + k}{k} }
  % \Psi_{B,A}(x_{B};x_{0}) 
  % =
  % \sum_{R \subseteq A \setminus C } \frac{(-1)^{| A | - |R| - k}}{ \binom{|R| + k}{k} }
  % \Psi_{B,A'}(x_{B};x_{0}),
    (-1)^{A'}|A'|\alpha_{A,B} = (-1)^{A}|A|\alpha_{A',B},
    \qquad\text{ whenever $B\subseteq A\cap A'$ and $|B| = k$}.
  % \sum_{l=0}^{|A'|-k} \frac{(-1)^{| A' | - l}}{ \binom{l + k}{k} }
  % \alpha_{A,B}
  % =
  % \sum_{l=0}^{|A|-k} \frac{(-1)^{| A | - l}}{ \binom{l + k}{k} }
  % \alpha_{A',B},\quad\text{ if $B\subseteq A\cap A'$ and $|B| = k$},
  \end{equation*}
  for all $x_{B}\in\Xcal_{B}$ and $x_{0}\in\Xcal_{0}$.
\end{lemma}
\begin{proof}
  Assume that the coefficients $\alpha_{A,B}$ of $(\kappa_{A})\in M_{k+1}$ satisfy the conditions stated in the lemma.
  We may multiply all functions $\Psi_{B}$ by scalars and assume
  \begin{align}
    \label{eq:Ktilde-alphas}
    \alpha_{A,B}=(-1)^{|A|-|B|},&\quad\text{ if }|B|<k,& \alpha_{A,B}=(-1)^{|A|-k}\frac{k}{|A|},&\quad\text{ if }|B|=k.
  \end{align}
  Then $\ln(\kappa_{C}(x_{C};x_{0}))$ equals the logarithm of the normalization constant plus
  \begin{align}
    % \ln(\kappa_{C}(x_{C};x_{0})) & = 
    \sum_{A\subseteq C}&\left( \sum_{\substack{B\subseteq A \\ |B| < k}}
      (-1)^{|A|-|B|} \Psi_{B}(x_{C}|_{B}; x_{0}) + \sum_{\substack{B\subseteq A \\ |B| = k}} (-1)^{|A|-k}\frac{k}{|A|}
      \Psi_{B}(x_{C}|_{B}; x_{0}) \right) \notag \\
    \displaybreak[0]
    & = \sum_{\substack{B\subseteq C\\ |B|< k}} \left(\sum_{R\subseteq C\setminus B}(-1)^{|R|}\right)
    \Psi_{B}(x_{C}|_{B}; x_{0})
    \notag\\ & \qquad\qquad\qquad
    + \sum_{\substack{B\subseteq C\\ |B|=k}}\left(\sum_{R\subseteq C\setminus B}(-1)^{|R|}\frac{k}{|R|+k}\right) \Psi_{B}(x_{C}|_{B}; x_{0}) \notag\\
    \displaybreak[0]
    & = \sum_{\substack{B\subseteq C\\ |B|< k}} \left(\sum_{l=0}^{|C|-|B|}(-1)^{l}\binom{|C|-|B|}{l}\right)
    \Psi_{B}(x_{C}|_{B}; x_{0})
    \notag\\ & \qquad\qquad\qquad
    + \sum_{\substack{B\subseteq C\\ |B|=k}}\left(\sum_{l=0}^{|C|-k}(-1)^{l}\binom{|C|-k}{l}\frac{k}{l+k}\right) \Psi_{B}(x_{C}|_{B}; x_{0})
    \displaybreak[0]
    \notag\\ & = \sum_{\substack{B\subseteq C\\ |B|< k}} \delta_{|C|,|B|} \Psi_{B}(x_{C}|_{B}; x_{0})
    % \\ & \qquad\qquad\qquad + \sum_{\substack{B\subseteq C\\ |B| = k}} \frac{k}{|C|\binom{|C|-1}{k-1}}
    % \Psi_{B}(x_{C}|_{B}; x_{0})\,.
    + \sum_{\substack{B\subseteq C\\ |B| = k}} \frac{1}{\binom{|C|}{k}} \Psi_{B}(x_{C}|_{B}; x_{0})\,,
    \label{eq:Ktilde-calculation}
  \end{align}
  where the identity $\sum_{i=0}^{r}\binom{r}{i}\frac{(-1)^{i}}{m+i}=1/\left((m+r)\binom{r+m-1}{m-1}\right)$ was used
  and $\delta_{a,b}$ denotes Kronecker's delta.  For $|C|>k$ the first sum is empty, and it follows that
  $\kappa_{C}$ satisfies the defining equality of~$\tilde M_{k+1}$.

  Conversely, if $(\kappa_{A})\in\tilde M_{k+1}$, then let $\alpha_{A,B}$ be as in~\eqref{eq:Ktilde-alphas}, and let
  \begin{equation*}
    \Psi_{B}(x_{B}; x_{0}) = \log(\kappa_{B}(x_{B}; x_{0})), \quad\text{ for all }x_{0}\in\Xcal_{0}, x_{B}\in\Xcal_{B}, |B|\le k\,.
  \end{equation*}
  These coefficients $\alpha_{A,B}$ and functions $\Psi_{B}$ together define an element $(\tilde\kappa_{A})\in M_{k+1}$.
  The calculation~\eqref{eq:Ktilde-calculation} shows that
  \begin{equation*}
    \tilde\kappa_{A}(x_{A}; x_{0}) =
    \begin{cases}
      \,\frac{1}{Z_{A,x_{A}}}\exp(\Psi_{A}(x_{A}; x_{0}) = \kappa_{A}(x_{A}; x_{0}), & \quad\text{ if }|A|\le k,\\
      \,\frac{1}{Z_{A,x_{A}}}\exp\left(\sum_{\substack{B\subseteq
            A\\|B|=k}}\frac{1}{\binom{|A|}{k}}\ln(\kappa_{B}(x_{A}|_{B};x_{0})\right), & \quad\text{ if }|A| > k,\\
    \end{cases}
  \end{equation*}
  and so~$(\kappa_{A})=(\tilde\kappa_{A})$ belongs to~$M_{k+1}$ and is of the desired form.
\end{proof}

\begin{thm}
  \label{thm:rob-kernel-k-interactions}
  Let $(\kappa_A)$ be functional modalities.  Then there exist functional
  modalities $(\tilde\kappa_{A})$ in the closure of $\tilde M_{k+1}$ such that the following holds: If $(\kappa_{A})$ is
  $k$-robust on a set $\Scal\subseteq\Xin$, then $\kappa_{A}(x|_{A}) = \tilde\kappa_{A}(x|_{A})$ for all $A\subseteq[n]$
% with $|A|\ge k$
  and all $x\in\Scal$.  In particular, $(\tilde\kappa_{A})$ belongs to the closure of the family of
  $(k+1)$-interactions.
\end{thm}

\begin{proof}
%  Assume first that $\kappa_{A}$ is strictly positive.
  Define $(\tilde\kappa_{A})$ via
  \begin{equation*}
    \tilde\kappa_{A}(x_{A}; x_{0}) =
    \begin{cases}
      \kappa_{A}(x_{A}; x_{0}), &\quad\text{ if }|A|\le k,\\
%      &\qquad\text{ or }\Rcal_{x}^{\min}(A)=\{A\}, \\
      \frac{1}{Z_{A,x_{A}}} \left(\prod_{\substack{B\subseteq A\\|B|=k}}\kappa_{B}(x_{A}|_{B};x_{0})\right)^{1/{\binom{|A|}{k}}}, &\quad\text{ else,}
    \end{cases}
  \end{equation*}
  where $Z_{A,x_{A}}$ is a normalization constant.  By definition, $(\tilde\kappa_{A})$ lies in the closure of~$\tilde
  M_{k+1}$.  Let $x\in\Scal$ and $C\subseteq[n]$.  If $|C|\le k$, then $\tilde\kappa_{C}(x|_{C}) = \kappa_{C}(x|_{C})$
  by definition of~$\tilde\kappa_{A}$.  So assume that $|C|>k$.  By definition of $k$-robustness, if $x\in\Scal$, then
  $\kappa_{C}(x|_{C}) = \kappa_{B}(x|_{B})$ for all $B\subset C$ with~$|B|=k$.  Therefore, if $x\in\Scal$ and~$|C|>k$,
  then
  \begin{equation*}
    \kappa_{C}(x|_{C};x_{0}) =  \left(\prod_{\substack{B\subseteq C\\|B|=k}}\kappa_{B}(x_{C}|_{B};x_{0})\right)^{1/{\binom{|C|}{k}}}.
  \end{equation*}
  Therefore, if $x\in\Scal$ and $|C|>k$, then $Z_{C,x|_{C}}=1$ and $\kappa_{C}(x|_{C})=\tilde\kappa_{C}(x|_{C})$.
\end{proof}

Since $M_{k+1}$ and $\tilde M_{k+1}$ are independent of $\Scal$, Theorem~\ref{thm:rob-kernel-k-interactions} shows
that these two families can be used to construct robust systems, when the set $\Scal$ % input distribution $\pin$
is not known a priori but must be learnt by the system, or when $\Scal$ % the input distribution is not constant over
changes with time and the system must adapt.

If we are not interested in all functional modalities but just the stochastic map $\kappa$ describing the unperturbed
system, we can describe $\kappa$ in terms of low interaction order.  The \emph{family of $(k+1)$-interaction stochastic
  maps}, denoted by~$K_{k+1}$, consists of all strictly positive maps $\kappa$ such that
\begin{equation*}
  \ln \kappa(x;x_{0}) = \sum_{\substack{A\subseteq[n] \\ |A|\le k}}\Psi_{A}(x|_{A};x_{0})
\end{equation*}
for some real functions $\Psi_{A}:\Xcal_{A}\to\Rb$.
\begin{cor}
  \label{cor:k-interaction-stochmap}
  Let $\kappa$ be a stochastic map.  For given $k$ there exists a stochastic map $\tilde\kappa$ in the closure of $K_{k+1}$
% the family of $(k+1)$-interactions stochastic maps
  such that the following holds: If $\kappa$ is $k$-robust on a set~$\Scal$,
  then $\kappa(x) = \tilde\kappa(x)$ for all~$x\in\Scal$.
\end{cor}
\begin{proof}
  If $\kappa$ is $k$-robust on~$\Scal$, there exist functional modalities $(\kappa_{A})_{A}$ with $\kappa=\kappa_{[n]}$.
  Choose $(\tilde\kappa_{A})$ as in Theorem~\ref{thm:rob-kernel-k-interactions}.  If~$x\in\Scal$, then
  $\kappa(x)=\kappa_{[n]}(x)=\tilde\kappa_{[n]}(x)$.  Hence the Corollary holds true with
  $\tilde\kappa=\tilde\kappa_{[n]}$.
\end{proof}

\begin{ex}
  The functional modalities~\eqref{eq:neuron-funcmod} do not lie in~$\tilde M_{2}$.  This does not mean that neural
  networks are not robust: In fact, it is possible to naturally redefine the functional
  modalities~\eqref{eq:neuron-funcmod} such that the new functional modalities lie in~$\tilde M_{2}$.

  The construction~\eqref{eq:neuron-funcmod} identifies the summand $w_{i}x_{i}x_{0}$ with~$\phi_{\{i\}}$.  Now we will
  make another identification: For each $i\in[n]$ let
%  $\Psi_{\{i\}}(x_{i}; x_{0}) := w_{i}x_{i}x_{0}$.
  \begin{equation*}
    \kappa_{\{i\}}(x_{i}; x_{0}) = \frac{1}{Z_{i,x_{i}}}\exp(n\,w_{i}\,x_{i}\,x_{0} - \eta)\,.
  \end{equation*}
  The unique extension of these stochastic maps to
  functional modalities $(\kappa_{A})$ in $M_{2}$ is given by
  \begin{equation}
    \label{eq:neuron-funcmod-rob}
    \kappa_{A}(x|_{A};x_{0}) = \frac{1}{Z_{A,x|_{A}}'}\left(\prod_{i\in A}\kappa_{\{i\}}(x_{i}; x_{0})\right)^{1/|A|}
    = \frac{1}{Z_{A,x|_{A}}}\exp\left(\frac{n}{|A|}\sum_{i\in A}w_{i}x_{i}x_{0} - \eta\right)\,,
  \end{equation}
  where $Z_{A,x|_{A}}$ and $Z_{A,x|_{A}}'$ are constants determined by normalization.  The functional modalities defined in this way
  lie in $\tilde M_{2}$, and the stochastic map $\kappa_{[n]}$ agrees with~\eqref{neuron2}.  Note that, by tuning the
  parameters $w_{1},\dots,w_{n}$, any combination of stochastic maps is possible for $\kappa_{1},\dots,\kappa_{n}$.
  This shows that any element of $\tilde M_{2}$ has a representation of the form~\eqref{eq:neuron-funcmod-rob}.

  As in example~\ref{ex:linear-threshold} we can scale the weights $w_{i}$ and the threshold $\eta$ by a factor of
  $\beta$ and send $\beta\to+\infty$.  This leads to the rule
  \begin{equation}
    \label{ex:robust-signum-modalities}
    \kappa_{A}(x_A; +1)
    \to \theta(\frac{n}{|A|}\sum_{i\in A} w_{i}x_{i} - \eta),
  \end{equation}
  which is a normalized variant of~\eqref{ex:signum-modalities}.

  The rule~\eqref{eq:neuron-funcmod-rob} implements a renormalization of the effect of the remaining inputs under
  knockout.  Similar renormalization procedures are sometimes used when training neural networks using Hebb's rule.
  Usually the total sum of the weights $\sum_{i}w_{i}$ is normalized to not grow to infinity.  The
  rule~\eqref{eq:neuron-funcmod-rob} suggests that under knockout all remaining weights are amplified by a common factor.
\end{ex}

The ideas leading to Theorem~\ref{thm:rob-kernel-k-interactions} can be applied to more general robustness
structures~$\Rcal$ as follows: For any $x\in\Xcal$ let
\begin{equation*}
  \Rcal_{x} :=
  \begin{cases}
    \Big\{ R\subseteq[n] : (R,x|_{R})\in\Rcal \Big\}, &\text{ if there exists }R\subseteq[n]\text{ with }(R,x|_{R})\in\Rcal, \\
    \big\{[n]\big\}, &\text{ else},
  \end{cases}
\end{equation*}
and let $\Rcal_{x}^{\min}$ be the subset of inclusion-minimal elements of~$\Rcal_{x}$.  If $(\kappa_{A})$ is
$\Rcal$-robust in $\Scal$, then
\begin{equation*}
  \kappa(x; x_{0}) = \kappa_{R}(x|_{R}; x_{0})\quad\text{ for any }R\in\Rcal_{x}^{\min}, \;x\in\Scal
\end{equation*}
and hence
\begin{equation*}
  \kappa(x; x_{0}) = \left(\prod_{R\in\Rcal_{x}^{\min}}\kappa_{R}(x|_{R}; x_{0})\right)^{1/|\Rcal_{x}^{\min}|}\,.
\end{equation*}

For any $C\subseteq[n]$ let $\Rcal_{x}^{\min}(C) = \{R\in\Rcal_{x}^{\min} : C\subseteq R\}$.  If $\Rcal$ is coherent,
then we can deduce
\begin{equation}
  \label{eq:geom-mean}
  \kappa_{C}(x|_{C}; x_{0}) = \left(\prod_{R\in\Rcal_{x}^{\min}(C)}\kappa_{R}(x|_{R}; x_{0})\right)^{1/|\Rcal_{x}^{\min}(C)|}
\end{equation}
for all $x\in\Scal$ with $\Rcal_{x}^{\min}(C)\neq\emptyset$.  This motivates the following definition: Denote by $\tilde
M_{\Rcal}$ the set of all strictly positive functional modalities that satisfy
\begin{equation*}
%  \label{eq:geom-mean}
  \kappa_{C}(x|_{C}; x_{0}) = \frac{1}{Z_{C,x|_{C}}}\left(\prod_{R\in\Rcal_{x}^{\min}(C)}\kappa_{R}(x|_{R}; x_{0})\right)^{1/|\Rcal_{x}^{\min}(C)|}
\end{equation*}
for all~$x\in\Xcal$ and all $C\subseteq[n]$ with $\Rcal_{x}^{\min}(C)\neq\emptyset$, where $Z_{C,x|_{C}}$ is a suitable
normalization constant.
%
% *** Importance of closure?
%
The same proof as for Theorem~\ref{thm:rob-kernel-k-interactions} implies:
\begin{thm}
  \label{thm:rob-kernel-means}
  Let $(\kappa_{A})$ be functional modalities, and assume that $\Rcal$ is coherent.  Then there exist functional
  modalities $(\tilde\kappa_{A})$ in the closure of $\tilde M_{\Rcal}$ such that the following holds: If $(\kappa_{A})$
  is $\Rcal$-robust on a set $\Scal\subseteq\Xcal$, then
  \begin{equation*}
    \kappa_{A}(x|_{A}) = \tilde\kappa_{A}(x|_{A})\quad\text{ for all }x\in\Scal.
% \text{ and }A\subseteq[n]\text{ with }\Rcal_{x}^{\min}(A)\neq\emptyset.
  \end{equation*}
\end{thm}
As a generalization of Lemma~\ref{lem:Ktilde-und-K-k+1}, we can also describe $\tilde M_{\Rcal}$ as a set of functional
modalities with limited interaction order.  To simplify the presentation, we assume that $\Rcal$ is \emph{saturated}, by
which we mean the following: If $(R,x_{R})\in\Rcal$ for some $x_{R}\in\Xcal_{R}$, then $(R,x_{R}')\in\Rcal$ for
all~$x_{R}\in\Xcal_{R}$.  In other words, a saturated robustness specification is given by enumerating a family of
subsets of~$[n]$.  For example, the robustness structures $\Rcal_{k}$ are saturated, while the robustness structures
defining canalyzing and nested canalyzing functions (see Section~\ref{sec:canalyzing-functions}) are not saturated.  If
$\Rcal$ is saturated, then $\Rcal_{x}$ and $\Rcal^{\min}_{x}$ are independent of~$x\in\Xcal$.

Consider the family
\begin{equation*}
  \Delta = \Big\{ C\subseteq[n] : C\subseteq R\text{ for some }R \in \Rcal_{x}^{\min}\text{ and }x\in\Xcal\Big\}\,,  % falls nicht saturated: Hier x\in\Scal
\end{equation*}
and let $\Delta(C)=\{R\in\Delta : R\subseteq C\}$.  Let $M_{\Delta}$ be the set of all functional modalities
$(\kappa_{A})$ such that there exist potentials $\phi_{A}$ of the form
\begin{equation}
  \label{eq:def-KDelta}
  \phi_{A}(x_{A};x_{0}) = \sum_{B\in\Delta(A)}\alpha_{A,B}\Psi_{B}(x_{A}|_{B};x_{0}),  % \substack{B\subseteq A \\ B\in\Delta}
\end{equation}
where $\alpha_{A,B}\in\Rb$ and $\Psi_{B}$ is an arbitrary function $\Rb^{\Xcal_{B}\times\Xcal_{0}}\to\Rb$.  We call
$M_{\Delta}$ the \emph{family of $\Delta$-interaction functional modalities}.  Note that the functions $\Psi_{B}$ do not
depend on~$A$.  This ensures a certain interdependence among the functional modalities~$\kappa_{A}$.
% Again, since $M_{\Delta}$ only contains strictly positive kernels, we are also interested in the closure of
% $M_{\Delta}$ with respect to the usual real topology on the space of matrices.

\begin{lemma}
  \label{lem:Ktilde-and-K}
  Assume that $\Rcal$ is coherent and saturated.  $\tilde M_{\Rcal}$ is a subset % of the closure
  of~$M_{\Delta}$.
\end{lemma}
\begin{proof}
  If $\Rcal_{x}=\emptyset$, % for some $x\in\Scal$,
  then $\Delta$ contains all sets.  The Möbius inversion formula shows
  that $M_{\Delta}$ contains all strictly positive functional modalities.  Therefore, we may assume that
  $\Rcal_{x}\neq\emptyset$. % for all~$x\in\Scal$.

  Define Gibbs potentials using the Möbius inversion~\eqref{eq:Moebiuspotentials}.  If $x\in\Scal$ and $A$ is large
  enough such that $\Rcal_{x}^{\min}(A)\neq\emptyset$, then
  \begin{align*}         
    \sum_{\substack{C \subseteq A \\ C\in\Rcal_{x}}} (-1)^{| A \setminus C |}   \ln \kappa_C (x|_C ; x_0 )       
    & = \sum_{\substack{C \subseteq A \\ C\in\Rcal_{x}}} (-1)^{| A \setminus C |} \,  \frac{1}{|\Rcal_{x}^{\min}(C)|}
    \sum_{B\in\Rcal_{x}^{\min}(C)} \ln \kappa_C (x|_C ; x_0 )     \\
    & = \sum_{\substack{C \subseteq A \\ C\in\Rcal_{x}}} (-1)^{| A \setminus C |} \,  \frac{1}{|\Rcal_{x}^{\min}(C)|}
    \sum_{B\in\Rcal_{x}^{\min}(C)} \ln \kappa_B (x|_B ; x_0 )     \\        
     = \sum_{B\in\Rcal_{x}^{\min}(C)}& \left\{ \sum_{R \subseteq A \setminus B } 
      (-1)^{| A | - |R| - k} \,  \frac{1}{|\Rcal_{x}^{\min}(B\cup R)|} \right\} \,  \ln \kappa_B (x|_B ; x_0 )\,.
  \end{align*}
  % Using standard formulas for binomial coefficients this can be reformulated to
  % \begin{align*}         
  %   \sum_{\substack{C \subseteq A \\ | C |  \geq k}} (-1)^{| A \setminus C |}   \ln \kappa_C (x|_C ; x_0 )       
  %   % \\
  %   % & = \sum_{\substack{B \subseteq A \\ |B| = k}} \left\{ \sum_{l=0}^{|A|-k}
  %   %   (-1)^{| A | - l - k} \,  \frac{\binom{|A|-k}{l}}{{ \binom{l + k}{k} } } \right\} \,  \ln \kappa_B (x|_B ; x_0 ) \\
  %   % & = \sum_{\substack{B \subseteq A \\ |B| = k}} \left\{ \sum_{l=0}^{|A|-k}
  %   %   (-1)^{| A | - l - k} \,  \frac{\binom{|A|}{k+l}}{\binom{|A|}{k}} \right\} \,  \ln \kappa_B (x|_B ; x_0 ) \\
  %   % & = \sum_{\substack{B \subseteq A \\ |B| = k}} \left\{ \frac{1}{\binom{|A|}{k}} \sum_{l=0}^{|A|-k}
  %   %   (-1)^{| A | - l - k} \,  \binom{|A|}{k+l} \right\} \,  \ln \kappa_B (x|_B ; x_0 ) \\
  %   % & = \sum_{\substack{B \subseteq A \\ |B| = k}} \left\{ \frac{1}{\binom{|A|}{k}} \sum_{l=0}^{|A|-k}
  %   %   (-1)^{l} \,  \binom{|A|}{l} \right\} \,  \ln \kappa_B (x|_B ; x_0 ) \\
  %   % & = \sum_{\substack{B \subseteq A \\ |B| = k}} \left\{ \frac{1}{\binom{|A|}{k}}
  %   %   (-1)^{|A|-k} \,  \binom{|A|-1}{|A|-k} \right\} \,  \ln \kappa_B (x|_B ; x_0 ) \\
  %   & = \sum_{\substack{B \subseteq A \\ |B| = k}} \left\{
  %     (-1)^{|A|-k} \,   \frac{k}{|A|} \right\} \,  \ln \kappa_B (x|_B ; x_0 )\,.
  %   % \\        
  %   % & = \sum_{\substack{B \subseteq A \\ | B | = k}}  \alpha_B  \,  \ln \kappa_B (x_B ; x_0 )    
  % \end{align*}
  Together with~\eqref{eq:Moebiuspotentials} this gives
  \begin{equation*}
    \phi_A (x|_A, x_0) 
    \; = \; \sum_{\substack{C \subseteq A \\ C \notin\Rcal_{x}}}  \alpha_{A,C} \,  \ln \kappa_C(x|_C ; x_0 )
    + \sum_{\substack{C \subseteq A \\ C\in\Rcal_{x}^{\min}}}  \alpha_{A,C} \,  \ln \kappa_C(x|_C ; x_0 )\, ,
  \end{equation*}
  where
  \begin{equation*}
    \alpha_{A,C} =
    \begin{cases}
      (-1)^{|A| - |C|}, & \text{ if } C\notin\Rcal_{x}, \\
%      (-1)^{|A|-k} \,   \frac{k}{|A|},
      \sum_{R \subseteq A \setminus B } 
        (-1)^{| A | - |R| - k} \,  \frac{1}{|\Rcal_{x}^{\min}(B\cup R)|}
      &  \text{ if } C\in\Rcal_{x}^{\min}\,.
      % \sum_{R \subseteq A \setminus C } 
      % (-1)^{| A | - |R| - k} \,  \frac{1}{{ \binom{|R| + k}{k} } }, &  \text{ if } |C|= k
    \end{cases}
  \end{equation*}
  This is clearly of the form~\eqref{eq:def-KDelta}.
\end{proof}

In the case $\Rcal=\Rcal_{k}$ the sum %
$\sum_{R\subseteq A\setminus B} (-1)^{|A|-|R|-k}\,\frac{1}{|\Rcal_{x}^{\min}(B\cup R)|}$ that appears in the proof of
Lemma~\ref{lem:Ktilde-and-K} can be solved explicitly, resulting in the statement of Lemma~\ref{lem:Ktilde-und-K-k+1}.
In the general case this is not possible.
% Since $M_{\Delta}$ and $\tilde M_{\Rcal}$ % $M_{k+1}$ and $\tilde M_{k+1}$
% are independent of $\Scal$, Proposition~\ref{prop:rob-kernel-means} and
% Lemma~\ref{lem:Ktilde-and-K} %~\ref{prop:rob-kernel-k-interactions}
% show that these two families can be used to construct robust systems, when the set $\Scal$ is not known a priori but
% must be learnt by the system, or when $\Scal$ changes with time and the system must adapt.

% Under the assumptions of Lemma~\ref{lem:Ktilde-and-K}, Corollary~\ref{cor:k-interaction-stochmap} also generalizes and
% shows that, if a stochastic map $\kappa$ is $\Rcal$-robust on a set~$\Scal$, then it can be described on $\Scal$ by the
% interactions in~$\Delta$.  We omit the details.
Corollary~\ref{cor:k-interaction-stochmap} also generalizes.  Let $\Delta$ be as above.  The set $K_{\Delta}$ of
\emph{$\Delta$-interactions stochastic maps} consists of all strictly positive stochastic maps $\kappa$ such that
\begin{equation*}
  \ln \kappa(x;x_{0}) = \sum_{\substack{A\in\Delta}}\Psi_{A}(x|_{A};x_{0})
\end{equation*}
for some real functions $\Psi_{A}:\Xcal_{A}\to\Rb$.
\begin{cor}
  \label{cor:interaction-stochmap}
  Let $\kappa$ be a stochastic map, and let $\Rcal$ be a coherent and saturated robustness specification.  There exists
  a stochastic map $\tilde\kappa$ in the closure of $K_{\Delta}$ % the family of $\Delta$-interactions stochastic maps
  such that the following holds: If $\kappa$ is $\Rcal$-robust on a set~$\Scal$, then $\kappa(x) = \tilde\kappa(x)$ for
  all~$x\in\Scal$.
\end{cor}
The proof is the same as the proof of Corollary~\ref{cor:k-interaction-stochmap}.
\begin{rem}
  \label{rem:one-for-all}
  Instead of representing functional modalities as a family $(\kappa_{A})$ of stochastic maps, it is possible to
  use a single stochastic map $\hat\kappa$, operating on a larger space, that integrates the information from the
  family~$(\kappa_{A})$.  The stochastic map $\hat\kappa$ can be constructed as follows: For each $i=1,\dots,n$ let
  $\hat\Xcal_{i}$ be the disjoint union of $\Xcal_{i}$ and one additional element, denoted by~$0$.  This additional
  state represents the knockout of $X_{i}$.  Let $\hat\Xcal_{\text{in}}=\hat\Xcal_{1}\times\dots\times\hat\Xcal_{n}$.
  For each $y\in\hat\Xcal_{\text{in}}$ let $\supp(y)=\{i: y_{i}\neq 0\}$.  We define the stochastic map
  $\hat\kappa:\Xcal_{0}\times\hat\Xcal_{\text{in}}$ via
  \begin{equation*}
    \hat\kappa(x; x_{0}) = \kappa_{\supp(x)}(x|_{\supp(x)}; x_{0}).
  \end{equation*}
  This construction gives a one-to-one correspondence between functional modalities and stochastic maps from $\hat\Xcal_{\text{in}}$
  to~$\Xcal_{0}$.
  % In the following we will mainly work with the family $(\kappa_{A})$ instead of the single kernel~$\hat\kappa$, since
  % we want to study

  As an example, consider the functional modalities defined in~\eqref{eq:neuron-funcmod}. In this example, the
  construction of $\hat\kappa$ is particularly easy: It just amounts to extending the input space to $\{-1,0,+1\}^{n}$.
  Equation~\eqref{neuron2} remains valid for~$\hat\kappa$.  The construction is more complicated for the functional
  modalities~\eqref{eq:neuron-funcmod-rob}.

  More generally, any Gibbs representation for functional modalities $(\kappa_{A})$ as in~\eqref{intpot} extends to a Gibbs
  representation of $\hat\kappa$: For any $B\subseteq[n]$, $x_{0}\in\Xcal_{0}$ and $x\in\hat\Xcal_{\text{in}}$ let
  \begin{equation*}
    \hat\phi_{B}(x,x_{0})=
    \begin{cases}
      \phi_{B}(x|_{B},x_{0}), & \qquad\text{ if }\supp(x)\subseteq B, \\
      0, & \qquad\text{ else.}
    \end{cases}
  \end{equation*}
  Then
  \begin{equation*}
    \hat\kappa(x; x_{0}) = \frac{e^{\sum_{B\subseteq[n]}\hat\phi_{B}(x,x_{0})}}{\sum_{x_{0}'\in\Xcal_{0}}e^{\sum_{B\subseteq[n]}\hat\phi_{B}(x,x_{0})}}.
  \end{equation*}
\end{rem}

\section{Robustness and conditional independence}
\label{sec:robustness-and-CI}

Given the probability distribution $\pin$ of the input variables and a stochastic map~$\kappa$ describing the system,
the joint probability distribution of the complete system can be computed from
\begin{equation*}
  p(x_{0},x) = \kappa(x;x_{0})\pin(x),\qquad\text{ for all }(x_{0},x)\in\Xtot,
\end{equation*}
As shown in Proposition~\ref{verb}, robustness of stochastic maps is related to conditional independence constraints on
the joint distribution.  In this section we study the set of all joint distributions that arise from robust systems in
this way.

% When the probability distribution $\pin$ of the input variables is known, the joint probability distribution of the
% complete system can be computed from~$\kappa$:
% \begin{equation*}
%   p(x_{0},x) = \kappa(x;x_{0})\pin(x),\qquad\text{ for all }(x_{0},x)\in\Xtot,
% \end{equation*}
% This construction was already used in the proof of Proposition~\ref{verb}.  Robustness of the stochastic map $\kappa$
% constrains the set of joint probability distributions that may appear.
Let $\Rcal$ be a robustness specification.  By Proposition~\ref{verb}, the stochastic map $\kappa$ is $\Rcal$-robust on
$\supp(\pin)$ if and only if for all $(R,x_R)\in\Rcal$ the output $X_{0}$ is (stochastically) independent of
$X_{[n]\setminus R}$, given that $X_{R}=x_{R}$.  In the following, this conditional independence (CI) statement will be
written as $\CI{X_{0}}{X_{[n]\setminus R}}[X_{R}=x_R]$.  This motivates the following definition: A joint distribution
$p$ is called \emph{$\Rcal$-robust} if it satisfies % the CI statements
%\begin{equation}
%  \label{eq:CI-collection}
$\CI{X_{0}}{X_{[n]\setminus R}}[X_{R}=x_{R}]$
%\end{equation}
for all $(R,x_{R})\in\Rcal$.  We denote by $\Pcal_{\Rcal}$ the set of all $\Rcal$-robust probability distributions.

The single conditional independence statement $\CI{X_{0}}{X_{[n]\setminus R}}[X_{R}=x_R]$ means that the conditional
distributions satisfy
\begin{multline*}
  p(X_{0}=x_{0}\;|\;\Xvin=x) = p(X_{0}=x_{0}\;|\;X_{R}=x_{R}),\quad\text{ for all }x\in\Xin\text{ with }p(x)>0\\
              \text{ and }x|_{R}=x_{R} \;.
\end{multline*}
It is often convenient to use another definition that avoids the need to work with conditional distributions: The
statement $\CI{X_{0}}{X_{[n]\setminus R}}[X_{R}=x_R]$ holds if and only if
\begin{equation}
  \label{eq:elementary-CI}
  p(x_{0},x_{S},x_{R})p(x_{0}',x_{S}',x_{R}) = p(x_{0},x_{S}',x_{R})p(x_{0}',x_{S},x_{R}),
\end{equation}
for all $x_{0},x_{0}'\in\Xcal_{0}, x_{S},x_{S}'\in\Xcal_{S}$ and $x_{R}\in\Xcal_{R}$.  Here, $p(x_{0},x_{S},x_{R})$ is
an abbreviation of $p(X_{0}=x_{0},X_{S}=x_{S},X_{R}=x_{R})$.  It is not difficult to see that these two definitions of
conditional independence are equivalent.  The formulation in terms of determinantal equations is used in algebraic
statistics~\cite{DrtonSturmfelsSullivant09:Algebraic_Statistics} and will also turn out to be useful here.

% \begin{ex}
% %  As before, let $\Rcal_{k}$ be the set of subsets of $[n]$ of cardinality $k$ or greater.
%   Let $\Rcal_{k}$ be the set of all pairs $(R,x_{R})$, where $R\subseteq[n]$ has cardinality $k$ or greater and $x_{R}\in\Xcal_{R}$.
%   In other words, a probability measure $p$ is $\Rcal_{k}$-robust, if we can knock out any $n-k$ input variables without
%   losing information on the output.
% %  For simplicity we will call such a probability distribution $k$-robust in the following.
% \end{ex}

A joint probability distribution $p$ can be written as a $d_{0}\times|\Xin|$-matrix.  Each
equation~\eqref{eq:elementary-CI} imposes conditions on this matrix saying that certain submatrices have rank one.  To
be precise, for any edge $(x,x')$ in the graph $G_{\Rcal}$ (defined in Section~\ref{sec:robust-canal})
equations~\eqref{eq:elementary-CI} for all $x_{0},x_{0}'\in\Xcal_{0}$ require that the submatrix
$(p_{kz})_{k\in\Xcal_{0},z\in\{x,x'\}}$ has rank one.
% More generally, if $K\subset G$ is a clique (i.e.~a complete subgraph), then the submatrix
% $(p_{kz})_{k\in\Xcal_{0},z\in K}$ has rank one.  This means that all columns of this submatrix are proportional.
%
% The vectors $\tilde p_{x}$, $x\in\Xin$, are the columns of $p$.  % We have shown:
For any $x\in\Xin$ denote by $\tilde p_{x}$ the vector with components $\tilde p_{x}(x_{0}) =
\mbox{$p(X_{0}=x_{0},\Xvin=x)$}$ for $x_{0}\in\Xcal_{0}$.  Then a distribution $p$ lies in $\Pcal_{\Rcal}$ if and only
if $\tilde p_{x}$ and $\tilde p_{y}$ are proportional for all edges $(x,y)$ of~$G_{\Rcal}$.  Observe that $\tilde p_{x}$
and $\tilde p_{y}$ are proportional if and only if either (i) one of $\tilde p_{x}$ and $\tilde p_{y}$ vanishes or (ii)
$\kappa(x)=\kappa(y)$.  This observation allows to reformulate the equivalence $(1)\Leftrightarrow(3)$ of
Proposition~\ref{verb} as follows:

% Even if the graph $G_{\Rcal}$ is connected, not all columns $\tilde p_{x}$ must be proportional to each other, since
% proportionality is not a transitive relation.  Instead, there are ``blocks'' of columns such that all columns within one
% block are proportional.
% %
% Let $G_{p}$ be the subgraph of $G_{\Rcal}$ induced by $\supp\tilde p:=\{x\in\Xin:\tilde p_{x}\neq
% 0\}$.  % , and denote by $\Bbf$ the set of connected components of $G_{p}$.
% We have shown:
\begin{lemma}
  \label{lem:robustprop}
  Let $\Scal=\{x\in\Xin:\tilde p_{x}\neq 0\}$.  A distribution $p$ lies in $\Pcal_{\Rcal}$ if and only if $\tilde p_{x}$
  and $\tilde p_{y}$ are proportional whenever $x,y\in\Scal$ lie in the same connected component of~$G_{\Rcal,\Scal}$.
\end{lemma}

% For any probability distribution $p$ on $\Xtot$, $x_{0}\in\Xcal_{0}$ and $x\in\Xin$ denote by $\tilde p_{x}$ the
% vector with components $\tilde p_{x}(x_{0}) = p(X_{0}=x_{0},\Xvin=x)$.
% Write $\supp\tilde p:=\{x\in\Xin:\tilde p_{x}\neq 0\}$.
For any family $\Bbf$ of subsets of $\Xin$ let $\Pcal_{\Bbf}$ be the set of probability distributions $p$ on $\Xtot$
that satisfy the following two conditions:
\begin{enumerate}
\item $\cup\Bbf = \{x\in\Xin:\tilde p_{x}\neq 0\}$,
\item $\tilde p_{x}$ and $\tilde p_{y}$ are proportional, whenever there exists $\Zcal\in\Bbf$ such that $x,y\in\Zcal$.
\end{enumerate}
Then $\Pcal_{\Rcal}=\bigcup_{\Bbf}\Pcal_{\Bbf}$, where the union is over all $\Rcal$-robustness structures~$\Bbf$.  The
disadvantage of this decomposition is that there are $\Rcal$-robustness structures $\Bbf$, $\Bbf'$ such that
$\Pcal_{\Bbf}$ is a subset of the topological closure $\ol{\Pcal_{\Bbf'}}$ of~$\Pcal_{\Bbf'}$.  In other words, each
$p\in\Pcal_{\Bbf}$ can be approximated arbitrarily well by elements of $\Pcal_{\Bbf'}$, and therefore in many cases it
suffices to only consider~$\Pcal_{\Bbf'}$.  The following definition is needed:
% It follows % from~\eqref{eq:decomposition-VG} and Theorem~\ref{thm:primary-decomposition}
% that
\begin{defi}
  An $\Rcal$-robustness structure $\Bbf$ is \emph{maximal} if and only if $\cup\Bbf:=\bigcup_{\Zcal\in\Bbf}\Zcal$
  satisfies any of the following equivalent conditions:
  \begin{enumerate}
  \item For any $x\in\Xin\setminus\cup\Bbf$ there are edges $(x,y)$, $(x,z)$ in $G_{\Rcal}$ such that $y,z\in\cup\Bbf$
    do not lie in the same connected component of $G_{\Rcal,\cup\Bbf}$.
  \item For any $x\in\Xin\setminus\cup\Bbf$ the induced subgraph $G_{\Rcal,\cup\Bbf\cup\{x\}}$ has fewer connected
    components than $G_{\Rcal,\cup\Bbf}$.
  \end{enumerate}
\end{defi}
\begin{lemma}
  \label{lem:decomposition}
  $\Pcal_{\Rcal}$ equals the disjoint union $\bigcup_{\Bbf}\Pcal_{\Bbf}$, where the union is over all % maximal
  $\Rcal$-ro\-bust\-ness structures.  Alternatively, $\Pcal_{\Rcal}$ equals the (non-disjoint) union
  $\bigcup_{\Bbf}\ol{\Pcal_{\Bbf}}$, where the union is over all maximal $\Rcal$-robustness structures.
\end{lemma}
\begin{proof}
  The first statement follows directly from the above considerations.  To see that it suffices to take maximal
  $\Rcal$-robustness structures in the second decomposition, consider an $\Rcal$-robustness structure $\Bbf$ that is not
  maximal.  By definition there exists $x\in\Xin\setminus\cup\Bbf$ such that the induced subgraph
  $G_{\Rcal,\cup\Bbf\cup\{x\}}$ has at least as many connected components as $G_{\Rcal,\cup\Bbf}$.
  % Denoting the connected components of $G_{\Rcal,\cup\Bbf\cup\{x\}}$ by $\Bbf'$, it is easy to see that
  % $\Pcal_{\Bbf}\subseteq\ol{\Pcal_{\Bbf'}}$.
  Let $\Bbf'$ be the family of connected components of $G_{\Rcal,\cup\Bbf\cup\{x\}}$.  If $G_{\Rcal,\cup\Bbf\cup\{x\}}$ has the same number of connected components as $G_{\Rcal,\cup\Bbf}$, then there is
  $\Ycal\in\Bbf$ such that $\Ycal\cup\{x\}\in\Bbf'$, otherwise let $\Ycal\in\Bbf$ be arbitrary.  Let $y\in\Ycal$.  For any $p\in\Pcal_{\Bbf}$ and $\epsilon>0$
  define a probability distribution $p_{\epsilon}$ via
  \begin{equation*}
    p_{\epsilon}(x_{0},z) =
    \begin{cases}
      p(x_{0},z), & \quad\text{ if }z\notin\{x,y\}, \\
      (1-\epsilon)p(x_{0},x), & \quad\text{ if }z=y, \\
      \epsilon p(x_{0},x), & \quad\text{ if }z=x.
    \end{cases}
  \end{equation*}
  Then $p_{\epsilon}\in\Pcal_{\Bbf'}$, and hence $\Pcal_{\Bbf}\subseteq\ol{\Pcal_{\Bbf'}}$.  If $\Bbf'$ is not maximal,
  we may iterate the process.
\end{proof}

% For any $x\in\Xin$ the vector $\tilde p_{x}$ is proportional to the conditional probability distribution
% $P(\,\cdot\,|\,\Xvin=x)$ of $X_{0}$ given that $\Xvin=x$.  Hence:
% \begin{lemma}
%   \label{lem:Crobustcrit}
%   Let $p$ be a probability distribution on $\Xtot$, and let $\Bbf$ be the set of connected components of
%   $G_{\Rcal,\supp\tilde p}$.  Then $p$ is $\Rcal$-robust if and only if $P(\,\cdot\,|\,\Xvin=x)=P(\,\cdot\,|\,\Xvin=y)$ whenever
%   there exists $\Zcal\in\Bbf$ such that $x,y\in\Zcal$.
%   % $B\in\Bbf$.
% \end{lemma}

% The set $\cup\Bbf$ % is the support of the input distribution~$\pin$.  It
% corresponds to the set $\Scal$ from Section~\ref{sec:robustn-mark-kern}.  Lemma~\ref{lem:Crobustcrit} gives a hint how
% to choose the set~$\Scal$: The goal is to have as many connected components as possible in~$G_{\Rcal,\Scal}$.
% We will derive bounds on the number of connected components in Section~\ref{sec:robust_functions}.

The following lemma sheds light on the structure of $\ol{\Pcal_{\Bbf}}$:
\begin{lemma}
  \label{lem:genconst}
  Fix an $\Rcal$-robustness structure $\Bbf$.  Then $\ol{\Pcal_{\Bbf}}$ consists of all probability measures of
  the form
  \begin{equation}
    \label{eq:genconst}
    p(X_{0}=x_{0}, \Xvin=x) =
    \begin{cases}
      \mu(\Zcal)\lambda_{\Zcal}(x) p_{\Zcal}(x_{0}), & \text{ if }x\in \Zcal\in\Bbf,
      \\
      0,                       & \text{ if }x\in\Xin\setminus\cup\Bbf,
    \end{cases}
  \end{equation}
  where $\mu$ is a probability distribution on $\Bbf$ and $\lambda_{\Zcal}$ is a probability distribution on $\Zcal$ for
  each $\Zcal\in\Bbf$ and $(p_{\Zcal})_{\Zcal\in\Bbf}$ is a family of probability distributions on $\Xcal_{0}$.
\end{lemma}
\begin{proof}
  It is easy to see that~\eqref{eq:genconst} defines indeed a probability distribution.  By Lemma~\ref{lem:robustprop}
  it belongs to $\Pcal_{\Bbf}$.  In the other direction, any probability measure % satisfies
  can be written as a product
  \begin{equation*}
    p(x_{0},x_{1},\dots,x_{n}) = p(\Zcal) p\left(x_{1},\dots,x_{n}\middle|(X_{1},\dots,X_{n})\in \Zcal\right) p(x_{0}|x_{1},\dots,x_{n}),
  \end{equation*}
  if $(x_{1},\dots,x_{n})\in \Zcal\in\Bbf$, and if $p$ is an $\Rcal$-robust probability distribution, then
%  $\lambda_{B}(j_{1},\dots,j_{n})=p\left(x_{1},\dots,x_{n}\middle|(X_{1},\dots,X_{n})\in B\right)$
  $p_{\Zcal}(x_{0}):=p(x_{0}|x_{1},\dots,x_{n})$ depends only on the block $\Zcal$ in which $(x_{1},\dots,x_{n})$ lies.
\end{proof}

Lemma~\ref{lem:decomposition} decomposes the set $\Pcal_{\Rcal}$ of robust probability distributions into the closures
of the smooth manifolds~$\Pcal_{\Bbf}$, where $\Bbf$ runs over the maximal $\Rcal$-robustness structures.
Lemma~\ref{lem:genconst} gives natural parametrizations of these manifolds.

By comparison, Theorem~\ref{thm:rob-kernel-means} and
Lemma~\ref{lem:Ktilde-and-K} % \ref{prop:rob-kernel-k-interactions}
describe robustness from a different point of view.
The result can be translated to the setting of this section as follows:
\begin{cor}
  \label{cor:p-n-kappa}
  Suppose that $\Rcal$ is a coherent and saturated robustness structure, and define $\Delta$ as in
  Section~\ref{sec:robustn-mark-kern}.  If $p\in\Pcal_{\Bbf}$, then there exists a stochastic map $\tilde\kappa\in
  K_{\Delta}$ such that $p(x_{0}|x) = \tilde\kappa(x; x_{0})$ for all~$x\in\cup\Bbf$.
\end{cor}
In the statement of the corollary note that $p(\Xin=x)>0$ for all~$x\in\cup\Bbf$, and hence the conditional
distribution $p(x_{0}|x)$ is well-defined in this case.

Corollary~\ref{cor:p-n-kappa} can also be viewed from the perspective of hierarchical models: Let $\tilde\Delta=\{\{1,\dots,n\}\} \cup
\{ S\cup\{0\}\,:\,S\in\Delta\}$.  
% To translate this result into the setting of this section, the following definition is needed: Let $\tilde\Delta$ be a
% family of subsets of $\{0,\dots,n\}$.
The \emph{hierarchical loglinear model} $\Ecal_{\tilde\Delta}$ consists of all probability distributions $p$ on $\Xtot$
of the form
\begin{equation*}
  \log(p(x)) = \sum_{A\subseteq\tilde\Delta}\tilde\phi_{A}(x|_{A}),
\end{equation*}
where $\phi_{A}$ is a real function with domain~$\Xcal_{A}$.  By the results of this section, $\Ecal_{\tilde\Delta}$ is
a smooth manifold containing $\Pcal_{\Rcal}$ in its closure.
See~\cite{Lauritzen96:Graphical_Models,DrtonSturmfelsSullivant09:Algebraic_Statistics} for more on hierarchical
loglinear models.

\begin{rem}
  \label{rem:algebra}
  It is also possible to derive the decomposition in Lemma~\ref{lem:decomposition} from results from commutative
  algebra.  Since the equations~\eqref{eq:elementary-CI} that describe conditional independence are algebraic, they
  generate a polynomial ideal, called \emph{conditional independence ideal}.  In this case the ideal is a
  \emph{generalized binomial edge ideal}, as defined in~\cite{Rauh12:Binomial_edge_ideals}.  For such ideals, the
  primary decomposition is known and corresponds precisely to the decomposition of the set of robust distributions as
  presented in Lemma~\ref{lem:decomposition}.
  The parametrization of Lemma~\ref{lem:genconst} can be considered as a surjective polynomial map and shows that all
  components of the decomposition are rational.
\end{rem}

\section{$k$-robustness} % {$\Rcal_{k}$-robustness}
\label{sec:k-robustness}

In this section we consider the symmetric case~$\Rcal=\Rcal_{k}$.  As above, we replace any prefix or subscript $\Rcal$
by $k$.

If $k=0$, then any pair $(x,y)$ is an edge in $G_{0}$.  This means that any $0$-robustness structure $\Bbf$ contains
only one set.  There is only one maximal $0$-robustness structure, namely $\overline{\Bbf}=\{\Xin\}$.  The set
$\Rcal_{0}$ is irreducible.  This corresponds to the fact that $\Pcal_{0}$ is defined by $\CI{X_{0}}{\Xvin}$.

$\overline{\Bbf}$ is actually a maximal $k$-robustness structure for any $0\le k< n$.  This illustrates the fact
that the single CI statement $\CI{X_{0}}{\Xvin}$ implies all other CI statements of the
form~$\CI{X_{0}}{X_{[n]\setminus R}}[X_{R}=x_{R}]$. % \eqref{eq:CI-collection}.
The corresponding set $\Pcal_{\overline{\Bbf}}$ contains all probability distributions of $\Pcal_{k}$ of full support.

Now let $k=1$.  In the case $n=2$ we obtain results by Alexander Fink, which can be reformulated as
follows~\cite{Fink11:Binomial_ideal_of_intersection_axiom}:
%
% \begin{thm}
%   \label{lem:n=2max}
\emph{  Let $n=2$.  A $1$-robustness structure $\Bbf$ is maximal if and only if the following statements hold:
  \begin{itemize}
  \item Each $B\in\Bbf$ is of the form $B=S_{1}\times S_{2}$, where $S_{1}\subseteq \Xcal_{1}, S_{2}\subseteq\Xcal_{2}$.
  \item For every $x_{1}\in\Xcal_{1}$ there exists $B\in\Bbf$ and $x_{2}\in\Xcal_{2}$ such that $(x_{1},x_{2})\in B$,
    and conversely.
  \end{itemize}
}% \end{cor}

In~\cite{Fink11:Binomial_ideal_of_intersection_axiom} a different description is given: The block $S_{1}\times S_{2}$ can
be identified with the complete bipartite graph on $S_{1}$ and $S_{2}$.  In this way, every maximal $1$-robustness
structure corresponds to a collection of complete bipartite subgraphs with vertices in $\Xcal_{1}\cup\Xcal_{2}$ such that
every vertex in $\Xcal_{1}$ and~$\Xcal_{2}$, respectively, is part of one such subgraph.  Figure~\ref{fig:Fink-example}
shows an example.
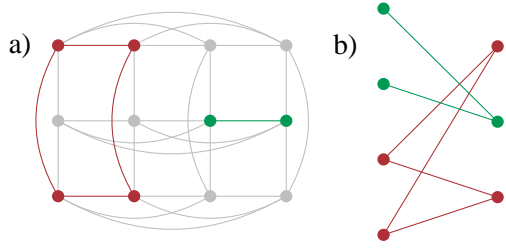
\begin{figure}
  \centering
  \begin{tikzpicture}
    \foreach \x in {0,...,3} { 
      \foreach \y in {0,1,2} { 
        \path (\x,\y) coordinate (X-\x-\y);
      }
    }
    \begin{scope}[lightgray]
%     \foreach \x in {0,1,2} { 
%       \foreach \y in {0,1} {
%         \pgfmathparse{\x+1}
%         \let\xp\pgfmathresult
%         \pgfmathparse{\y+1}
%         \let\yp\pgfmathresult
%         \node (X-\x-\y) { \xp , \yp };
%         \fill (X-\x-\y) circle (2.3pt); 
% %        \draw (X-\x-\y) -- (X-\xp-\yp);
%       }}
      \foreach \x in {0,...,3} {
        \draw (X-\x-0) -- (X-\x-1) -- (X-\x-2);
      }
      \draw (X-1-0) -- (X-2-0) -- (X-3-0);
      \draw (X-0-1) -- (X-1-1) -- (X-2-1);
      \draw (X-1-2) -- (X-2-2) -- (X-3-2);
      \draw (X-2-0) to [out=120,in=-120] (X-2-2);
      \draw (X-3-0) to [out=60,in=-60] (X-3-2);
      \draw (X-0-0) to [out=-30,in=-150] (X-2-0);
      \draw (X-1-0) to [out=-30,in=-150] (X-3-0);
      \draw (X-0-0) to [out=-30,in=-150] (X-3-0);
      \draw (X-0-1) to [out=-30,in=-150] (X-2-1);
      \draw (X-1-1) to [out=-30,in=-150] (X-3-1);
      \draw (X-0-1) to [out=-30,in=-150] (X-3-1);
      \draw (X-0-2) to [out=30,in=150] (X-2-2);
      \draw (X-1-2) to [out=30,in=150] (X-3-2);
      \draw (X-0-2) to [out=30,in=150] (X-3-2);
      \foreach \i in {2-0,3-0,2-2,3-2,0-1,1-1} { \fill (X-\i) circle (2.3pt); }
    \end{scope}
    \draw[Maroon] (X-0-0) -- (X-1-0) to [out=120,in=-120] (X-1-2) -- (X-0-2) to [out=-120,in=120] (X-0-0);
    \draw[ForestGreen] (X-2-1) -- (X-3-1);
    \foreach \i in {0-0,1-0,0-2,1-2} { \fill[Maroon] (X-\i) circle (2.3pt); }
    \foreach \i in {2-1,3-1} { \fill[ForestGreen] (X-\i) circle (2.3pt); }
    \node at (-0.5,2) { a) };
  \end{tikzpicture}
  \begin{tikzpicture}
    \path (0,0) coordinate (X0);
    \path (0,1) coordinate (X1);
    \path (0,2) coordinate (X2);
    \path (0,3) coordinate (X3);
    \path (1.5,0.5) coordinate (Y0);
    \path (1.5,1.5) coordinate (Y1);
    \path (1.5,2.5) coordinate (Y2);
    \foreach \i in {0,1} {
      \foreach \j in {0,2} { \draw[Maroon] (X\i) -- (Y\j); } }
    \draw[ForestGreen] (X2) -- (Y1) -- (X3);
    \foreach \i in {0,1} { \fill[Maroon] (X\i) circle (2.3pt); }
    \foreach \i in {2,3} { \fill[ForestGreen] (X\i) circle (2.3pt); }
    \foreach \i in {0,2} { \fill[Maroon] (Y\i) circle (2.3pt); }
    \foreach \i in {1} { \fill[ForestGreen] (Y\i) circle (2.3pt); }
    \node at (-0.5,2.5) { b) };    
  \end{tikzpicture}
  \caption{A 1-robustness structures for two variables. a) The graph $G_{1,\Scal}$.  b) The representation in terms of
    bipartite graphs.}
  \label{fig:Fink-example}
\end{figure}

This result generalizes in the following way:
\begin{lemma}
  A $1$-robustness structure $\Bbf$ is maximal if and only if the following statements hold:
  \begin{itemize}
  \item Each $B\in\Bbf$ is of the form $B=S_{1}\times\dots\times S_{n}$, where $S_{i}\subseteq\Xcal_{i}$.
  \item $\bigcup_{S_{1}\times\dots\times S_{n}\in\Bbf}S_{i}=\Xcal_{i}$ for all $i\in[n]$
  \end{itemize}
\end{lemma}
\begin{proof}
%  We say that a subset $\Ycal$ of $\Xin$ is connected if $G_{\Rcal,\Ycal}$ is connected.
  Suppose that $\Bbf$ is maximal.  Let $\Ycal\in\Bbf$ and let $S_{i}$ be the projection of $\Ycal\subset\Xin$
  to~$\Xcal_{i}$.  Let $\Ycal'=S_{1}\times\dots\times S_{n}$.  Then $\Ycal\subseteq \Ycal'$.  We claim that
  $(\Bbf\setminus\{\Ycal\})\cup\{\Ycal'\}$ is another 1-robustness structure with the same number of components
  as~$\Bbf$, and by maximality we can conclude $\Ycal=\Ycal'$.
%contradicting our assumption that $\Bbf$ was maximal.
  By Definition~\ref{def:robustness-structure} we need to show that $G_{\Rcal,\Ycal'}$ is connected and that
  $G_{\Rcal,\Zcal\cup \Ycal'}$ is not connected for all $\Zcal\in\Bbf\setminus\{\Ycal\}$.  The first condition follows
  from the fact that $G_{\Rcal,\Ycal}$ is connected.
  For the second condition assume to the contrary that there are $x\in \Ycal'$ and $y\in \Zcal$ such that
  $x=(x_{1},\dots,x_{n})$ and $y = (y_{1},\dots,y_{n})$ disagree in at most $n-1$ components.  Then there exists a
  common component $x_{l}=y_{l}$.  By construction there exists $z=(z_{1},\dots,z_{n})\in \Ycal$ such that
  $z_{l}=y_{l}=x_{l}$, hence $\Ycal\cup \Zcal$ is connected, in contradiction to the assumptions.  This shows that each $\Ycal$ has
  a product structure.

  Write $\Ycal=S^{\Ycal}_{1}\times\dots\times S^{\Ycal}_{n}$ for each $\Ycal\in\Bbf$.  Obviously $S^{\Ycal}_{i}\cap
  S^{\Zcal}_{i}=\emptyset$ for all $i\in[n]$ and all $\Ycal,\Zcal\in\Bbf$ if $\Ycal\neq\Zcal$.  For the second
  assertion, assume to the contrary that $l\in\Xcal_{i}$ is contained in no~$S^{\Ycal}_{i}$.  Take any $\Ycal\in\Bbf$
  and define $\Ycal':= S^{\Ycal}_{1}\times\dots\times (S^{\Ycal}_{i}\cup\{l\})\times\dots\times S^{\Ycal}_{n}$.  Then
  $\left(\Bbf\setminus\{\Ycal\}\right)\cup\{\Ycal'\}$ is another $1$-robustness structure with the same number of
  components as~$\Bbf$, contradicting the assumptions.

  Conversely, assume that $\Bbf$ is a $1$-robustness structure satisfying the two assertions of the theorem.  For any
  $x\in\Xin\setminus\cup\Bbf$ there exist $y_{1},\dots,y_{n}\in\cup\Bbf$ such that $x_{1}=y_{1}$,\dots,$x_{n}=y_{n}$.
  Since $x\notin\cup\Bbf$ the points $y_{1},\dots,y_{n}$ cannot all belong to the same block of~$\Bbf$.  If $y_{i}$ and
  $y_{j}$ belong to different blocks of~$\Bbf$, then the two edges $(x,y_{i})$ and $(x,y_{j})$ of $G_{1}$ show that
  $\Bbf$ is maximal.
\end{proof}

The last result can be reformulated in terms of $n$-partite graphs
generalizing~\cite{Fink11:Binomial_ideal_of_intersection_axiom}: Namely, the $1$-robustness structures are in one-to-one
relation with the $n$-partite subgraphs of $M_{d_{1},\dots,d_{n}}$ such that every connected component is itself a
complete $n$-partite subgraph $M_{e_{1},\dots,e_{n}}$ with $e_{i}>0$ for all $i\in[n]$.  Here, an $n$-partite graph is a
graph which can be coloured by $n$ colours such that no two vertices with the same colour are adjacent.

Unfortunately the nice product form of the maximal $1$-robustness structures does not generalize to $k>1$:
\begin{ex}[Binary inputs]
  \label{ex:bin-inputs}
  If $n=3$ and $d_{1}=d_{2}=d_{3}=2$, then the graph $G_{2}$ is the graph of the cube. For a maximal 2-robustness
  structure $\Bbf$ the set $\Xin\setminus\cup\Bbf$ can be any one of
%  collection of separating sets is given by
  the following (see Fig.~\ref{fig:bin-3-inputs}): % collection of cuts
  \begin{itemize}
  \item The empty set
  \item A set of cardinality 4 corresponding to a plane leaving two connected components of size 2
  \item A set of cardinality 4 containing all vertices with the same parity.
% which are simplices disconnecting the remaining 4 points.
  \item A set of cardinality 3 cutting off a vertex.
  \end{itemize}
  In the last case
  % \begin{equation*}
  %   % \Bbf := \left\{\{(i,j,k): i,j,k\in\{1,2\}\}, \{(3,3,3),(3,3,2),(3,2,3),(2,3,3)\}\right\}
  %   \Bbf := \left\{\{(1,1,1)\}, \{(2,2,2),(2,2,1),(2,1,2),(1,2,2)\}\right\}.
  % \end{equation*}
  only the isolated vertex has a product structure (Fig.~\ref{fig:3-robustness}d).

  If $n=4$ and $d_{1}=d_{2}=d_{3}=d_{4}=2$, then the graph $G_{3}$ is the graph of a hyper-cube.
  Figure~\ref{fig:3-robustness} shows how a maximal 3-robustness structure can look like.
\end{ex}
\begin{figure}
  \centering

  \begin{tikzpicture}[scale=0.5]
    \path (1,0) coordinate (X0);
    \path (3,0) coordinate (X1);
    \path (1,2) coordinate (X2);
    \path (3,2) coordinate (X3);
    \path (1.5,0.3) coordinate (X4);
    \path (3.5,0.3) coordinate (X5);
    \path (1.5,2.3) coordinate (X6);
    \path (3.5,2.3) coordinate (X7);
    \draw (X0) -- (X1) -- (X3) -- (X2) -- (X0) -- (X4) -- (X5) -- (X7) -- (X6) -- (X4);
    \draw (X3) -- (X7);
    \draw (X1) -- (X5);
    \draw (X2) -- (X6);
    \foreach \i in {0,...,7} { \fill (X\i) circle (4.6pt); }
    \node at (0.3,2) { a) };
  \end{tikzpicture}
  $ $
  %% the following pictures have the right coordinates, if we don't scale again ==> scale the radii
  \begin{tikzpicture}
    \begin{scope}[lightgray]
      \draw (X0) -- (X1) -- (X3) -- (X2) -- (X0) -- (X4) -- (X5) -- (X7) -- (X6) -- (X4);
      \draw (X3) -- (X7);
      \foreach \i in {0,3,4,7} { \fill (X\i) circle (2.3pt); }
    \end{scope}
    \draw (X1) -- (X5);
    \draw (X2) -- (X6);
    \foreach \i in {1,2,5,6} { \fill (X\i) circle (2.3pt); }
    \node at (0.15,1) { b) };
  \end{tikzpicture}
  $ $
  \begin{tikzpicture}
    \begin{scope}[lightgray]
      \draw (X0) -- (X1) -- (X3) -- (X2) -- (X0) -- (X4) -- (X5) -- (X7) -- (X6) -- (X4);
      \draw (X3) -- (X7);
      \draw (X1) -- (X5);
      \draw (X2) -- (X6);
      \foreach \i in {0,3,5,6} { \fill (X\i) circle (2.3pt); }
    \end{scope}
    \foreach \i in {1,2,4,7} { \fill (X\i) circle (2.3pt); }    
    \node at (0.15,1) { c) };
  \end{tikzpicture}
  $ $
  \begin{tikzpicture}
    \begin{scope}[lightgray]
      \draw (X1) -- (X3) -- (X2);
      \draw (X5) -- (X7) -- (X6) -- (X4) -- (X5);
      \draw (X3) -- (X7);
      \draw (X1) -- (X5);
      \draw (X2) -- (X6);
      \foreach \i in {3,5,6} { \fill (X\i) circle (2.3pt); }
    \end{scope}
    \draw (X1) -- (X0) -- (X2);
    \draw (X0) -- (X4);
    \foreach \i in {0,1,2,4,7} { \fill (X\i) circle (2.3pt); }    
    \node at (0.15,1) { d) };
  \end{tikzpicture}

  \caption{The four symmetry classes of maximal 2-robustness structures of three binary inputs, see
    Example~\ref{ex:bin-inputs}.}
  \label{fig:bin-3-inputs}
\end{figure}
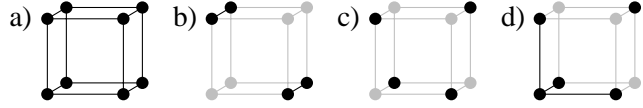
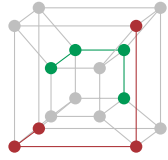
\begin{figure}
  \centering
  \begin{tikzpicture}[scale=0.8]
    \fourcubevert
    \begin{scope}[ForestGreen]
      \draw (X13) -- (X15) -- (X14) -- (X10);
    \end{scope}
    \begin{scope}[lightgray]
      \draw (X2) -- (X0);
      \draw (X4) -- (X5) -- (X7) -- (X6) -- (X4);
      \draw (X3) -- (X7);
      \draw (X1) -- (X5);
      \draw (X2) -- (X6);
      \draw (X8) -- (X9) -- (X11) -- (X10) -- (X8) -- (X12) -- (X13);
      \draw (X14) -- (X12);
      \draw (X11) -- (X15);
      \draw (X9) -- (X13);
      \draw (X0) -- (X8);
      \draw (X1) -- (X9);
      \draw (X10) -- (X2) -- (X3) -- (X11);
      \draw (X4) -- (X12);
      \draw (X5) -- (X13);
      \draw (X6) -- (X14);
      \draw (X7) -- (X15);
      \foreach \i in {5,...,9} { \fill (X\i) circle (2.875pt); }
      \foreach \i in {2,11,12} { \fill (X\i) circle (2.875pt); }
    \end{scope}
    \begin{scope}[ForestGreen]
      \foreach \i in {10,13,14,15} { \fill (X\i) circle (2.875pt); }
    \end{scope}
    \begin{scope}[Maroon]
      \draw (X4) -- (X0) -- (X1) -- (X3);
      \foreach \i in {0,1,3,4} { \fill (X\i) circle (2.875pt); }
    \end{scope}
  \end{tikzpicture}
  
  \caption{A maximal 3-robustness structure for four binary inputs.}
  \label{fig:3-robustness}
\end{figure}

% Generically, the smaller $k$, the easier it is to describe the structure of all $k$-robustness structures.  We have seen
% above that the cases $k=0$ and $k=1$ are particularly nice.
$k$-robustness implies $(k+1)$-robustness, and therefore $\Pcal_{k}\subseteq\Pcal_{k+1}$.
% One might expect that all $k$-robustness structures are also $(k+1)$-robustness structures for all $k$.
% A positive statement in this direction is Lemma \ref{lem:smallk}.
% Unfortunately, this is not true in general,
This does not mean that all $k$-robustness structures are also $(k+1)$-robustness structures, for the following reason:
If $\Bbf$ is a $k$-robustness structure and $\Scal=\cup\Bbf$, then $G_{k+1,\Scal}$ may have more connected components
than~$G_{k,\Scal}$.
\begin{ex}
  \label{ex:2-not-3}
  Consider $n=4$ binary random variables $X_{1},\dots,X_{4}$.  Then
  \begin{equation*}
    \Bbf := \left\{ \{(1,1,1,1),(2,2,1,1)\}, \{(1,2,2,2), (2,1,2,2)\} \right\}
  \end{equation*}
  is a maximal $2$-robustness structure.  Both elements of $\Bbf$ are connected in $G_{2}$, but not in~$G_{3}$, see
  Fig.~\ref{fig:2-not-3}.
\end{ex}
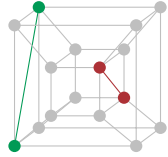
\begin{figure}
  \centering
  \begin{tikzpicture}[scale=0.8]
    \fourcubevert
    % \foreach \i in {0,...,8} { \draw (X\i) -- (X\i+8); }
    \begin{scope}[ForestGreen]
      \draw (X0) -- (X6);
    \end{scope}
    \begin{scope}[lightgray]
      \draw (X0) -- (X1) -- (X3) -- (X2) -- (X0) -- (X4) -- (X5) -- (X7) -- (X6) -- (X4);
      \draw (X3) -- (X7);
      \draw (X1) -- (X5);
      \draw (X2) -- (X6);
      \draw (X8) -- (X9) -- (X11) -- (X10) -- (X8) -- (X12) -- (X13) -- (X15) -- (X14) -- (X12);
      \draw (X11) -- (X15);
      \draw (X9) -- (X13);
      \draw (X10) -- (X14);
      \draw (X0) -- (X8);
      \draw (X1) -- (X9);
      \draw (X2) -- (X10);
      \draw (X3) -- (X11);
      \draw (X4) -- (X12);
      \draw (X5) -- (X13);
      \draw (X6) -- (X14);
      \draw (X7) -- (X15);
      \foreach \i in {1,...,5} { \fill (X\i) circle (2.875pt); }
      \foreach \i in {7,...,10} { \fill (X\i) circle (2.875pt); }
      \foreach \i in {12,14,15} { \fill (X\i) circle (2.875pt); }
    \end{scope}
    \begin{scope}[ForestGreen]
      \foreach \i in {0,6} { \fill (X\i) circle (2.875pt); }
    \end{scope}
    \begin{scope}[Maroon]
      \draw (X11) -- (X13);
      \foreach \i in {11,13} { \fill (X\i) circle (2.875pt); }
    \end{scope}
  \end{tikzpicture}
  
  \caption{The $2$-robustness structure from Example~\ref{ex:2-not-3}.  The graph $G_{2}$ is the graph of a hypercube of
    dimension four, where diagonals have been added to the two-dimensional faces.  Only the edges of $G_{2}$ that
    connect vertices of Hamming distance one are shown, and the edges of~$G_{2,\cup\Bbf}$.  The two blocks are marked in
    green and red.}
  \label{fig:2-not-3}
\end{figure}

Nevertheless, the notions of $l$-robustness and $k$-robustness for $l>k$ are related as follows:

\begin{lemma}
  \label{lem:smallk}
  Assume that $d_{1}=\dots=d_{n}=2$, and let $\Bbf$ be a maximal $k$-robustness structure of binary random
  variables.  Then each $B\in\Bbf$ is connected as a subset of $G_{s}$ for all $s\le n - 2 k + 1$.
\end{lemma}
\begin{proof}
  We can identify elements of $\Xin$ with binary strings of length $n$.  Denote by $I_{r}$ the string $1\dots10\dots0$ of
  $r$ ones and $n-r$ zeroes in this order.  Without loss of generality assume that $I_{0}, I_{l}$ are two elements of
  $B\in\Bbf$, where $k \le n-l < s \le n - 2k + 1$.  Then $l\ge 2k$, and hence $\lfloor\frac l2\rfloor\ge k$.
  % Assume $l$
  Let $m = \lceil\frac{l}{2}\rceil$.  We will prove that we can replace $B$ by $B\cup\{I_{m}\}$ and obtain another
  $k$-robustness structure.  By maximality this will imply that $I_{0}$ and $I_{l}$ are indeed connected by a
  path in $G_{s}$.

  Otherwise there exists $A\in\Bbf$, $A\neq B$, and $x\in A$ such that $x$ and $I_{m}$ agree in at least $k$ components.
  Let $a$ be the number of zeroes in the first $m$ components of $x$, let $b$ be the number of ones in the components
  from $m+1$ to $l$ and let $c$ be the number of ones in the last $n-l$ components.  Then $I_{m}$ and $x$ disagree in
  $a+b+c \le n-k$ components.  On the other hand, $x$ and $I_{0}$ disagree in $(m-a) + b + c$ components, and $x$ and
  $I_{l}$ disagree in $a + ((l-m)-b) + c \le a + (m-b) + c$ components.  Assume that $a\ge b$ (otherwise exchange
  $I_{0}$ and $I_{l}$ in the following argument).  Then $x$ and $I_{0}$ disagree in at most $m + c \le \lceil\frac{l}{2}\rceil + n - l = n -
  \lfloor\frac{l}{2}\rfloor \le n-k$ components, so $A\cup B$ is connected, in contradiction to the assumptions.
\end{proof}

% \begin{rem}
%   Let $(\kappa_{A})$ be functional modalities that are $k$-robust on a set~$\Scal$.  For any input distribution $\pin$
%   supported on $\Scal$ the corresponding joint distribution $p$ is also $k$-robust.
% \end{rem}

\subsection*{Acknowledgement}

This work has been supported by the Volkswagen Foundation and the Santa Fe Institute. Nihat Ay thanks David Krakauer and
Jessica Flack for many stimulating discussions on robustness.
% We thank the reviewers for very detailed reports.

\bibliographystyle{IEEEtranS}
\bibliography{general}

\end{document}